\documentclass[11pt, notitlepage]{article}   

\usepackage{amsmath,amsthm,amsfonts}   

\usepackage{amscd}
\usepackage{amsfonts}
\usepackage{amssymb}
\usepackage{color}      
\usepackage{epsfig}
\usepackage{graphicx}           

\usepackage{tikz}

\theoremstyle{plain}
\newtheorem{theorem}{Theorem}[section]
\newtheorem{lemma}[theorem]{Lemma}
\newtheorem{corollary}[theorem]{Corollary}
\newtheorem{proposition}[theorem]{Proposition}
\newtheorem{observation}[theorem]{Observation}
\newtheorem{remark}[theorem]{Remark}

\newtheorem{question}[theorem]{Question}

\theoremstyle{definition}

\newcommand{\code}{\textnormal{code}}

\def\finf{\mathop{{\rm I}\kern -.27 em {\rm F}}\nolimits}

\newcommand{\edim}{\textnormal{edim}}

\begin{document}


\title{On the edge dimension and fractional edge dimension of graphs}

\author{{\bf{Eunjeong Yi}}\\
\small Texas A\&M University at Galveston, Galveston, TX 77553, USA\\
{\small\em yie@tamug.edu}}

\maketitle

\date{}

\begin{abstract}
Let $G$ be a graph with vertex set $V(G)$ and edge set $E(G)$, and let $d(u,w)$ denote the length of an $u-w$ geodesic in $G$. For any vertex $v\in V(G)$ and any edge $e=xy\in E(G)$, let $d(e,v)=\min\{d(x,v), d(y,v)\}$. For any distinct edges $e_1, e_2\in E(G)$, let $R\{e_1,e_2\}=\{z\in V(G): d(z,e_1) \neq d(z,e_2)\}$. Kelenc, Tratnik and Yero [Discrete Appl. Math. 251 (2018) 204-220] introduced the notion of an edge resolving set and the edge dimension of a graph: A vertex subset $S\subseteq V(G)$ is an \emph{edge resolving set} of $G$ if $|S \cap R\{e_1, e_2\}|\ge 1$ for any distinct edges $e_1, e_2\in E(G)$, and the \emph{edge dimension}, $\edim(G)$, of $G$ is the minimum cardinality among all edge resolving sets of $G$.

For a function $g$ defined on $V(G)$ and for $U\subseteq V(G)$, let $g(U)=\sum_{s\in U}g(s)$. A real-valued function $g: V(G) \rightarrow [0,1]$ is an \emph{edge resolving function} of $G$ if $g(R\{e_1, e_2\})\ge1$ for any distinct edges $e_1, e_2\in E(G)$. The \emph{fractional edge dimension}, $\edim_f(G)$, of $G$ is $\min\{g(V(G)): g \mbox{ is an edge resolving function of }G\}$. Note that $\edim_f(G)$ reduces to $\edim(G)$ if the codomain of edge resolving functions is restricted to $\{0,1\}$.

In this paper, we introduce and study the fractional edge dimension of graphs, and we obtain some general results on the edge dimension of graphs. We show that there exist two non-isomorphic graphs on the same vertex set with the same edge metric coordinates. We construct two graphs $G$ and $H$ such that $H \subset G$ and both $\edim(H)-\edim(G)$ and $\edim_f(H)-\edim_f(G)$ can be arbitrarily large. We show that a graph $G$ with $\edim(G)=2$ cannot have $K_5$ or $K_{3,3}$ as a subgraph, and we construct a non-planar graph $H$ satisfying $\edim(H)=2$. It is easy to see that, for any connected graph $G$ of order at least three, $1 \le \edim_f(G) \le \frac{|V(G)|}{2}$; we characterize graphs $G$ satisfying $\edim_f(G)=1$ and examine some graph classes satisfying $\edim_f(G)=\frac{|V(G)|}{2}$. We also determine the fractional edge dimension for some classes of graphs.
\end{abstract}



\section{Introduction}

Let $G$ be a finite, simple, undirected, and connected graph with vertex set $V(G)$ and edge set $E(G)$. For any two vertices $x, y\in V(G)$, let $d(x,y)$ denote the minimum number of edges connecting the vertices $x$ and $y$ in $G$. For $v\in V(G)$, the \emph{open neighborhood} of $v$ is $N(v)=\{u\in V(G): uv\in E(G)\}$. The \emph{degree} of a vertex $v\in V(G)$ is $|N(v)|$; a \emph{leaf} is a vertex of degree one, and a \emph{major vertex} is a vertex of degree at least three. The \emph{complement} of $G$, denoted by $\overline{G}$, is the graph whose vertex set is $V(G)$ and $xy\in E(\overline{G})$ if and only if $xy\not\in E(G)$ for any distinct $x,y\in V(G)$. We denote by $P_n$, $C_n$, $K_n$ and $K_{t, n-t}$, respectively, the path, the cycle, the complete graph and the complete bi-partite graph on $n$ vertices.

A vertex $z \in V(G)$ \emph{resolves} a pair of vertices $x$ and $y$ in $G$ if $d(x,z) \neq d(y,z)$. For two distinct vertices $x,y \in V(G)$, let $R_v\{x,y\}=\{z \in V(G): d(x,z) \neq d(y,z)\}$. A vertex subset $S \subseteq V(G)$ is a \emph{(vertex) resolving set} of $G$ if $|S\cap R_v\{x,y\}| \ge 1$ for every pair of distinct vertices $x,y \in V(G)$, and the \emph{metric dimension} $\dim(G)$ of $G$ is the minimum cardinality among all resolving sets of $G$. For an ordered set of distinct vertices $U=\{u_1, \ldots, u_k\} \subseteq V(G)$, the distance vector of a vertex $x\in V(G)$ with respect to $U$ is $\code_U(x)=(d(x, u_1), \ldots, d(x, u_k))$. Metric dimension, introduced by  Slater~\cite{slater} and by Harary and Melter~\cite{harary}, has applications in robot navigation~\cite{tree2}, sonar~\cite{slater} and combinational optimization~\cite{sebo}, to name a few. It is noted in~\cite{NP} that determining the metric dimension of a general graph is an NP-hard problem.

For any vertex $v\in V(G)$ and any edge $e=xy\in E(G)$, let $d(e,v)=\min\{d(x,v), d(y,v)\}$. For any distinct edges $e_1, e_2 \in E(G)$, let $R_e\{e_1, e_2\}=\{v\in V(G): d(v,e_1) \neq d(v, e_2)\}$. A vertex subset $S \subseteq V(G)$ is an \emph{edge resolving set} of $G$ if $|S \cap R_e\{e_1, e_2\}| \ge 1$ for every pair of distinct edges $e_1, e_2\in E(G)$, and the \emph{edge (metric) dimension} $\edim(G)$ of $G$ is the minimum cardinality among all edge resolving sets of $G$. Kelenc et al.~\cite{edim} introduced and initiated the study of edge dimension, and it is stated in~\cite{edim} that determining the edge dimension of a general graph is an NP-complete problem. It is easy to see that, for any connected graph $G$ of order $n\ge3$, $1\le \edim(G)\le n-1$ (see~\cite{edim}); for characterization of graphs $G$ with $\edim(G)=n-1$ and $\edim(G)=n-2$, respectively, see~\cite{nina} and~\cite{geneson}. For an ordered set of distinct vertices $U=\{u_1, \ldots, u_k\} \subseteq V(G)$, the distance vector of an edge $e\in E(G)$ with respect to $U$ is $\code_U(e)=(d(e, u_1), \ldots, d(e, u_k))$.

The fractionalization of various graph parameters has been extensively studied (see~\cite{fractionalization}). Currie and Oellermann~\cite{oellermann} defined fractional metric dimension as the optimal solution to a linear programming problem by relaxing a condition of the integer programming problem for metric dimension. Arumugam and Mathew~\cite{fdim} officially studied the fractional metric dimension of graphs. For a function $g$ defined on $V(G)$ and for $U\subseteq V(G)$, let $g(U)=\sum_{s\in U} g(s)$. A real-valued function $g:V(G)\rightarrow[0,1]$ is a \emph{(vertex) resolving function} of $G$ if $g(R_v\{x,y\}) \ge 1$ for every pair of distinct vertices $x,y\in V(G)$, and the \emph{fractional metric dimension} $\dim_f(G)$ of $G$ is $\min\{g(V(G)): g \mbox{ is a resolving function of } G\}$. Notice that $\dim_f(G)$ reduces to $\dim(G)$ if the codomain of resolving functions is restricted to $\{0,1\}$.

Analogous to resolving function and fractional metric dimension, we introduce edge resolving function and fractional edge (metric) dimension as follows. A real-valued function $g: V(G) \rightarrow [0,1]$ is an \emph{edge resolving function} of $G$ if $g(R_e\{e_1, e_2\})\ge1$ for every pair of distinct edges $e_1, e_2\in E(G)$, and the \emph{fractional edge dimension} $\edim_f(G)$ of $G$ is $\min\{g(V(G)): g \mbox{ is an edge resolving function of }G\}$. Notice that $\edim_f(G)$ reduces to $\edim(G)$ if the codomain of edge resolving functions is restricted to $\{0,1\}$.

In this paper, we introduce and study the fractional edge dimension of graphs. We obtain some general results on the edge dimension and fractional edge dimension of graphs. We also determine the fractional edge dimension of some graph classes.

The paper is organized as follows. In Section~\ref{sec_general}, we observe that $1\le\edim_f(G)\le \frac{|V(G)|}{2}$ for any connected graph $G$ of order at least three. We characterize connected graphs $G$ satisfying $\edim_f(G)=1$, and we show that $\dim_f(G)=\frac{|V(G)|}{2}$ implies $\edim_f(G)=\frac{|V(G)|}{2}$, but not vice versa; we show that there exists a family of graphs $H$ with $\edim_f(H)=\frac{|V(H)|}{2}>\dim_f(H)$ such that $\frac{\edim_f(H)}{\dim_f(H)}$ can be arbitrarily large. We show that, for an edge resolving set $S$ of $G$, $\{\code_S(e): e\in E(G)\}$ does not uniquely determine $G$, i.e., there exist two non-isomorphic graphs on the same vertex set with the same edge metric coordinates with respect to the same edge resolving set. We show that there exist graphs $G$ and $H$ with $H \subset G$ such that both $\edim(H)-\edim(G)$ and $\edim_f(H)-\edim_f(G)$ can be arbitrarily large. We also consider the relation between planarity of a graph $G$ and $\edim(G)=2$: we show that $\edim(G)=2$ implies $G$ contains neither $K_5$ nor $K_{3,3}$ as a subgraph, while there exists a non-planar graph $H$ with $\edim(H)=2$. In Section~\ref{sec_graphs}, we determine $\edim_f(G)$ when $G$ is a tree, a cycle, the Petersen graph, a wheel graph, a complete multi-partite graph and a grid graph, respectively. 


\section{General results on edge dimension and fractional edge dimension}\label{sec_general}

In this section, we obtain some general results on edge dimension and fractional edge dimension. We begin with some terminology and useful observations. Two vertices $u,w\in V(G)$ are called \emph{twins} if $N(u)-\{w\}=N(w)-\{u\}$; notice that a vertex is its own twin. Hernando et al.~\cite{Hernando} observed that the twin relation is an equivalence relation and that an equivalence class under it, called a \emph{twin equivalence class}, induces either a clique or an independent set. We note that, for distinct twins $x$ and $y$ of $G$, if $z\in N(x) \cap N(y)$, then $R_e\{zx, zy\}=\{x,y\}$.

\begin{observation}\label{obs_twin}
Let $x$ and $y$ be distinct members of the same twin equivalence class of a graph $G$. 
\begin{itemize}
\item[(a)] \emph{\cite{Hernando}} For any resolving set $R$ of $G$, $R \cap\{x,y\} \neq\emptyset$.
\item[(b)] For any edge resolving set $S$ of $G$, $S \cap\{x,y\} \neq\emptyset$.
\item[(c)] \emph{\cite{yi}} For any resolving function $g$ of $G$, $g(x)+g(y)\ge1$. 
\item[(d)] For any edge resolving function $h$ of $G$, $h(x)+h(y)\ge1$.
\end{itemize}
\end{observation}

\begin{observation}\label{obs_frac}
Let $G$ be a connected graph of order at least three. Then
\begin{itemize}
\item[(a)] \emph{\cite{fdim}} $\dim_f(G) \le \dim(G)$;
\item[(b)] $\edim_f(G) \le \edim(G)$.
\end{itemize}
\end{observation}

First, we show that there exist two non-isomorphic graphs on the same vertex set with the same edge metric coordinates with respect to the same edge resolving set. Seb\"{o} and Tannier~\cite{sebo} observed that, for a minimum resolving set $S$ of a graph $G$, the vectors $\{\code_S(v): v\in V(G)\}$ may not uniquely determine $G$ (see Figure~\ref{fig_unique}(a)). Similarly, we show that there exist two non-isomorphic graphs $H_1$ and $H_2$ with $V(H_1)=V(H_2)$ and $\{\code_S(e): e\in E(H_1)\}=\{\code_S(e'): e'\in E(H_2)\}$, where $S$ is a common minimum edge resolving set for $H_1$ and $H_2$; see Figure~\ref{fig_unique}(b).

\begin{figure}[ht]
\centering
\begin{tikzpicture}[scale=.7, transform shape]


\node [draw, fill=black, shape=circle, scale=.8] (a1) at  (0,0) {};
\node [draw, shape=circle, scale=.8] (a2) at  (0,1.5) {};
\node [draw, shape=circle, scale=.8] (a3) at  (0,3) {};
\node [draw, fill=black, shape=circle, scale=.8] (a4) at  (1.5,0) {};
\node [draw, shape=circle, scale=.8] (a5) at  (1.5,1.5) {};
\node [draw, shape=circle, scale=.8] (a6) at  (1.5,3) {};

\node [draw, fill=black, shape=circle, scale=.8] (a11) at  (4,0) {};
\node [draw, shape=circle, scale=.8] (a22) at  (4,1.5) {};
\node [draw, shape=circle, scale=.8] (a33) at  (4,3) {};
\node [draw, fill=black, shape=circle, scale=.8] (a44) at  (5.5,0) {};
\node [draw, shape=circle, scale=.8] (a55) at  (5.5,1.5) {};
\node [draw, shape=circle, scale=.8] (a66) at  (5.5,3) {};

\draw(a1)--(a2)--(a3);\draw(a4)--(a5)--(a6);\draw(a2)--(a5);\draw(a11)--(a22)--(a33)--(a66)--(a55)--(a44);\draw(a22)--(a55);

\node [scale=1] at (-0.6,3) {$(2,3)$};
\node [scale=1] at (-0.6,1.5) {$(1,2)$};
\node [scale=1] at (-0.6,0) {$(0,3)$};
\node [scale=1] at (2.1,3) {$(3,2)$};
\node [scale=1] at (2.1,1.5) {$(2,1)$};
\node [scale=1] at (2.1,0) {$(3,0)$};

\node [scale=1] at (3.4,3) {$(2,3)$};
\node [scale=1] at (3.4,1.5) {$(1,2)$};
\node [scale=1] at (3.4,0) {$(0,3)$};
\node [scale=1] at (6.1,3) {$(3,2)$};
\node [scale=1] at (6.1,1.5) {$(2,1)$};
\node [scale=1] at (6.1,0) {$(3,0)$};

\node [scale=1.3] at (0.75,-0.75) {$G_1$};
\node [scale=1.3] at (4.75,-0.75) {$G_2$};


\node [draw, fill=black, shape=circle, scale=.8] (b1) at  (10,3) {};
\node [draw, shape=circle, scale=.8] (b2) at  (9,2.2) {};
\node [draw, fill=black, shape=circle, scale=.8] (b3) at  (9,0.9) {};
\node [draw, shape=circle, scale=.8] (b4) at  (10,0) {};
\node [draw, shape=circle, scale=.8] (b5) at  (11,0.9) {};
\node [draw, shape=circle, scale=.8] (b6) at  (11,2.2) {};

\node [draw, fill=black, shape=circle, scale=.8] (b11) at  (14.5,3) {};
\node [draw, shape=circle, scale=.8] (b22) at  (13.5,2.2) {};
\node [draw, fill=black, shape=circle, scale=.8] (b33) at  (13.5,0.9) {};
\node [draw, shape=circle, scale=.8] (b44) at  (14.5,0) {};
\node [draw, shape=circle, scale=.8] (b55) at  (15.5,0.9) {};
\node [draw, shape=circle, scale=.8] (b66) at  (15.5,2.2) {};

\draw(b1)--(b2)--(b3)--(b4)--(b5)--(b6)--(b1);\draw(b2)--(b6);\draw(b11)--(b22)--(b33)--(b44)--(b55)--(b66)--(b11);\draw(b22)--(b55);

\node [scale=1] at (9.1,2.8) {$(0,1)$};
\node [scale=1] at (10.9,2.8) {$(0,2)$};
\node [scale=1] at (10,1.9) {$(1,1)$};
\node [scale=1] at (8.5,1.55) {$(1,0)$};
\node [scale=1] at (11.5,1.55) {$(1,2)$};
\node [scale=1] at (9,0.3) {$(2,0)$};
\node [scale=1] at (11,0.3) {$(2,1)$};

\node [scale=1] at (13.6,2.8) {$(0,1)$};
\node [scale=1] at (15.4,2.8) {$(0,2)$};
\node [scale=1] at (14.55,1.95) {$(1,1)$};
\node [scale=1] at (13,1.55) {$(1,0)$};
\node [scale=1] at (16,1.55) {$(1,2)$};
\node [scale=1] at (13.5,0.3) {$(2,0)$};
\node [scale=1] at (15.5,0.3) {$(2,1)$};

\node [scale=1.3] at (10,-0.75) {$H_1$};
\node [scale=1.3] at (14.5,-0.75) {$H_2$};

\node [scale=1.3] at (2.75,-1.5) {\textbf{(a)}};
\node [scale=1.3] at (12.25,-1.5) {\textbf{(b)}};

\end{tikzpicture}
\caption{\small{(a)~\cite{sebo} Two non-isomorphic graphs $G_1$ and $G_2$ with $V(G_1)=V(G_2)$ and $\{\code_S(v):v\in V(G_1)\}=\{\code_S(w):w\in V(G_2)\}$ on the common minimum resolving set $S$, comprised of the solid vertices, for $G_1$ and $G_2$; (b) Two non-isomorphic graphs $H_1$ and $H_2$ with $V(H_1)=V(H_2)$ and $\{\code_S(e): e\in E(H_1)\}=\{\code_S(e'):e'\in E(H_2)\}$ on the common minimum edge resolving set $S$, comprised of the solid vertices, for $H_1$ and $H_2$.}}\label{fig_unique}
\end{figure}
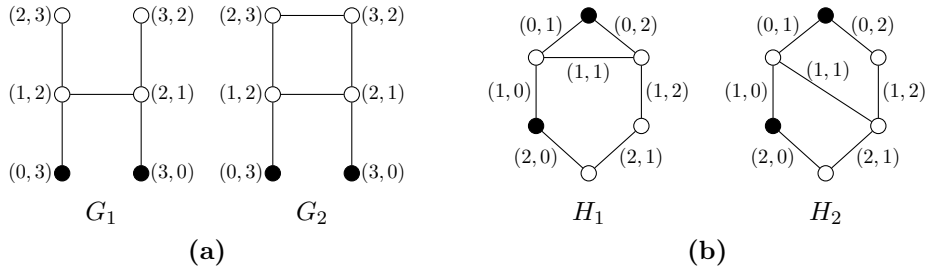

Second, we examine the relation between $\edim(G)=2$ and the planarity of $G$. We recall some terminology. A graph is \emph{planar} if it can be drawn in a plane without edge crossing. For two graphs $G$ and $H$, $H$ is called a \emph{minor} of $G$ if $H$ can be obtained from $G$ by vertex deletion, edge deletion, or edge contraction. We recall the following results. 

\begin{theorem}\emph{\cite{wagner}}\label{graph_minor}
A graph $G$ is planar if and only if neither $K_5$ nor $K_{3,3}$ is a minor of $G$.
\end{theorem}

\begin{theorem}\emph{\cite{tree2}}\label{dim2_planarity}
\begin{itemize}
\item[(a)] A graph $G$ with $\dim(G)=2$ cannot have $K_5$ or $K_{3,3}$ as a subgraph.
\item[(b)] There exists a non-planar graph $G$ with $\dim(G)=2$.
\end{itemize}
\end{theorem}

Analogous to Theorem~\ref{dim2_planarity}, we show that $\edim(G)=2$ implies $G$ contains neither $K_5$ nor $K_{3,3}$ as a subgraph, while there exists a non-planar graph $H$ with $\edim(H)=2$.

\begin{theorem}
If $G$ is a graph with $\edim(G)=2$, then $G$ contains neither $K_5$ nor $K_{3,3}$ as a subgraph.
\end{theorem}

\begin{proof}
Let $S=\{w_1, w_2\} \subset V(G)$ be an edge resolving set of $G$ with $|S|=2$.

First, suppose $G$ contains a clique with 5 vertices $v_1, v_2, v_3, v_4, v_5$. By relabeling the vertices of $G$ if necessary, let $d(w_1, v_1)=\min\{d(w_1, v_i): 1 \le i \le 5\}$. Then $\code_{\{w_1\}}(v_1v_2)=\code_{\{w_1\}}(v_1v_3)=\code_{\{w_1\}}(v_1v_4)=\code_{\{w_1\}}(v_1v_5)$. Since $v_1,v_2,v_3,v_4,v_5$ are all adjacent, there exist two edges, say $v_1v_i$ and $v_1v_j$ for distinct $i,j\in\{2,3,4,5\}$, with $\code_S(v_1v_i)=\code_S(v_1v_j)$, contradicting the assumption that $S$ is an edge resolving set of $G$. So, if $\edim(G)=2$, then $G$ does not contain $K_5$ as a subgraph.

Second, suppose $G$ contains $K_{3,3}$ as a subgraph; let $V(K_{3,3})=V_1\cup V_2$ such that each vertex in $V_1=\{u_1, u_2, u_3\}$ is adjacent to each vertex in $V_2=\{v_1, v_2, v_3\}$. Note that, for each $i,j\in\{1,2,3\}$, $\code_{\{w_1\}}(u_iv_j) \in \{\alpha, \alpha+1\}$ and $\code_{\{w_2\}}(u_iv_j) \in \{\beta, \beta+1\}$ for some non-negative integers $\alpha$ and $\beta$; thus, there are only four possible values for $\code_S(e)$ for $e\in E(K_{3,3})$. Since $|E(K_{3,3})|=9$, there exist two distinct edges, say $e_1$ and $e_2$, with $\code_S(e_1)=\code_S(e_2)$, contradicting the assumption that $S$ is an edge resolving set of $G$. So, if $\edim(G)=2$, then $G$ does not contain $K_{3,3}$ as a subgraph.~\hfill
\end{proof}

\begin{theorem}
There exists a non-planar graph $G$ with $\edim(G)=2$.
\end{theorem}

\begin{proof}
Let $G$ be the graph given in Figure~\ref{fig_edim2_nonplanar}. Since $G$ contains $K_{3,3}$ as a minor, $G$ is not planar by Theorem~\ref{graph_minor}. Next, we show that $S=\{x_1, y_4\}$ forms an edge resolving set of $G$. Note the following: $\code_S(u_1x_1)=(0,2)$, $\code_S(x_1x_2)=(0,3)$, $\code_S(x_2v_1)=(1,4)$, $\code_S(u_1v_2)=(1,1)$, $\code_S(u_1v_3)=(1,2)$, $\code_S(u_2v_1)=(2,4)$, $\code_S(u_2y_1)=(3,3)$, $\code_S(y_1y_2)=(4,2)$, $\code_S(y_2y_3)=(4,1)$, $\code_S(y_3y_4)=(3,0)$, $\code_S(y_4v_2)=(2,0)$, $\code_S(u_2v_3)=(2,3)$, $\code_S(u_3z_1)=(3,2)$, $\code_S(z_1z_2)=(4,3)$, $\code_S(z_2z_3)=(3,4)$, $\code_S(z_3v_1)=(2,5)$, $\code_S(u_3v_2)=(2,1)$ and $\code_S(u_3v_3)=(2,2)$. So, $S$ is an edge resolving set of $G$ with $|S|=2$, and thus $\edim(G) \le 2$. Since a set consisting of one vertex fails to form an edge resolving set of $G$, $\edim(G)=2$.~\hfill
\end{proof}

\begin{figure}[ht]
\centering
\begin{tikzpicture}[scale=.7, transform shape]

\node [draw, shape=circle, scale=.8] (u1) at  (0,6) {};
\node [draw, shape=circle, scale=.8] (u2) at  (0,3) {};
\node [draw, shape=circle, scale=.8] (u3) at  (0,0) {};

\node [draw, shape=circle, scale=.8] (v1) at  (10,6) {};
\node [draw, shape=circle, scale=.8] (v2) at  (10,3) {};
\node [draw, shape=circle, scale=.8] (v3) at  (10,0) {};

\node [draw, fill=black, shape=circle, scale=.8] (u11) at  (3.3,6) {};
\node [draw, shape=circle, scale=.8] (u12) at  (6.6,6) {};

\node [draw, shape=circle, scale=.8] (u21) at  (2,3) {};
\node [draw, shape=circle, scale=.8] (u22) at  (4,3) {};
\node [draw, shape=circle, scale=.8] (u23) at  (6,3) {};
\node [draw, fill=black, shape=circle, scale=.8] (u24) at  (8,3) {};

\node [draw, shape=circle, scale=.8] (u31) at  (1.3,0.8) {};
\node [draw, shape=circle, scale=.8] (u32) at  (2.5,1.5) {};
\node [draw, shape=circle, scale=.8] (u33) at  (7.5,4.5) {};

\draw(u1)--(u11)--(u12)--(v1)--(u2)--(v3);\draw(v2)--(u1)--(v3)--(u3)--(v2);\draw(u2)--(u21)--(u22)--(u23)--(u24)--(v2);\draw(u3)--(u31)--(u32)--(u33)--(v1); 

\node [scale=1.2] at (-0.4,6) {$u_1$};
\node [scale=1.2] at (-0.4,3) {$u_2$};
\node [scale=1.2] at (-0.4,0) {$u_3$};
\node [scale=1.2] at (10.4,6) {$v_1$};
\node [scale=1.2] at (10.4,3) {$v_2$};
\node [scale=1.2] at (10.4,0) {$v_3$};

\node [scale=1.1] at (3.3,6.35) {$x_1$};
\node [scale=1.1] at (6.6,6.35) {$x_2$};

\node [scale=1.1] at (2,3.35) {$y_1$};
\node [scale=1.1] at (4,3.35) {$y_2$};
\node [scale=1.1] at (6,3.35) {$y_3$};
\node [scale=1.1] at (8,3.35) {$y_4$};

\node [scale=1.1] at (1.3, 1.15) {$z_1$};
\node [scale=1.1] at (2.5,1.85) {$z_2$};
\node [scale=1.1] at (7.6,4.85) {$z_3$};

\end{tikzpicture}
\caption{\small{A non-planar graph $G$ with $\edim(G)=2$, where $\{x_1, y_4\}$ is an edge resolving set of $G$}.}\label{fig_edim2_nonplanar}
\end{figure}
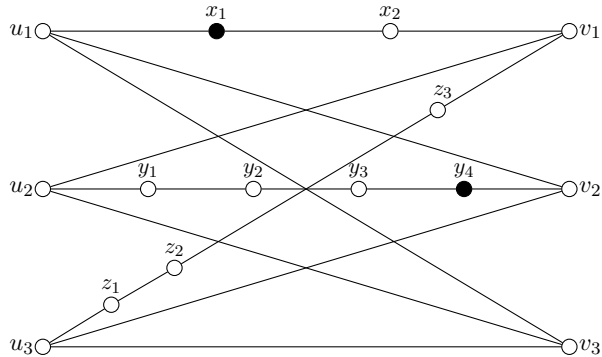

Third, we show that there exist graphs $G$ and $H$ with $H \subset G$ such that both $\edim(H)-\edim(G)$ and $\edim_f(H)-\edim_f(G)$ can be arbitrarily large; along the way, we show that $\dim_f(H)-\dim_f(G)$ can be arbitrarily large. We begin by recalling the characterization result of graphs $G$ satisfying $\dim_f(G)=\frac{|V(G)|}{2}$ that are used in proving our results. 

\begin{theorem}\emph{\cite{fdim, frac_kang}}\label{fracdim_bounds}
Let $G$ be a connected graph of order $n\ge2$. Then $1 \le \dim_f(G)\le \frac{n}{2}$, and $\dim_f(G)=\frac{n}{2}$ if and only if there exists a bijection $\phi:V(G) \rightarrow V(G)$ such that $\phi(v)\neq v$ and $|R\{v, \phi(v)\}|=2$ for all $v\in V(G)$.
\end{theorem}

For an explicit characterization of graphs $G$ satisfying $\dim_f(G)=\frac{|V(G)|}{2}$, we recall the following construction from~\cite{fracdim2}. Let $\mathcal{K}=\{K_{s}: s \ge 2\}$ and $\overline{\mathcal{K}}=\{\overline{K}_t: t \ge 2\}$. Let $H[\mathcal{K} \cup \overline{\mathcal{K}}]$ be the family of graphs obtained from a connected graph $H$ by replacing each vertex $u_i \in V(H)$ by a graph $H_i \in \mathcal{K} \cup \overline{\mathcal{K}}$, and each vertex in $H_i$ is adjacent to each vertex in $H_j$ if and only if $u_iu_j \in E(H)$.

\begin{theorem}~\emph{\cite{fracdim2}}\label{n_2}
Let $G$ be a connected graph of order at least two. Then $\dim_f(G)=\frac{|V(G)|}{2}$ if and only if $G \in H[\mathcal{K} \cup \overline{\mathcal{K}}]$ for some connected graph $H$.
\end{theorem}

Now, we recall that metric dimension is not a monotone parameter on subgraph inclusion (see~\cite{lineG}). We recall the following stronger result from~\cite{broadcast}.

\begin{theorem}\emph{\cite{broadcast}}
There exist connected graphs $G$ and $H$ such that $H \subset G$ and $\frac{\dim(H)}{\dim(G)}$ can be arbitrarily large.
\end{theorem}

Following~\cite{broadcast}, for $m \ge 3$, let $H_m=K_{\frac{m(m+1)}{2}}$; let $V(H_m)$ be partitioned into $V_1, V_2, \ldots, V_m$ such that $V_i=\{w_{i,1}, w_{i,2}, \ldots, w_{i,i}\}$ with $|V_i|=i$, where $i \in \{1,2,\ldots,m\}$. Let $G_m$ be the graph obtained from $H_m$ and $m$ isolated vertices $u_1,u_2, \ldots, u_m$ such that, for each $i \in \{1,2,\ldots, m\}$, $u_i$ is joined by an edge to each vertex of $V_i \cup (\cup_{j=i+1}^{m}\{w_{j,i}\})$. See Figure~\ref{fig_subgraph} when $m=4$. Note that $H_m\subset G_m$, $\dim_f(H_m) =\frac{m(m+1)}{4}$ by Theorem~\ref{n_2}, and $\dim_f(G_m) \le m$ by Observation~\ref{obs_frac}(a) since $\{u_1, u_2, \ldots, u_m\}$ forms a resolving set of $G_m$ (see~\cite{broadcast}). So, $\frac{\dim_f(H_m)}{\dim_f(G_m)} \ge \frac{m+1}{4} \rightarrow \infty$ as $m\rightarrow \infty$. Thus, we have the following corollary.

\begin{figure}[ht]
\centering
\begin{tikzpicture}[scale=.7, transform shape]

\node [draw, shape=circle, scale=.8] (0) at  (0, 0) {};
\node [draw, shape=circle, scale=.8] (1) at  (1.5, 0) {};
\node [draw, shape=circle, scale=.8] (2) at  (3, 0) {};
\node [draw, shape=circle, scale=.8] (3) at  (4.5, 0) {};
\node [draw, shape=circle, scale=.8] (4) at  (6, 0) {};
\node [draw, shape=circle, scale=.8] (5) at  (7.5, 0) {};
\node [draw, shape=circle, scale=.8] (6) at  (9, 0) {};
\node [draw, shape=circle, scale=.8] (7) at  (10.5, 0) {};
\node [draw, shape=circle, scale=.8] (8) at  (12, 0) {};
\node [draw, shape=circle, scale=.8] (9) at  (13.5, 0) {};

\node [draw, shape=circle, scale=.8] (a) at  (0, -3) {};
\node [draw, shape=circle, scale=.8] (b) at  (2.25, -3) {};
\node [draw, shape=circle, scale=.8] (c) at  (6, -3) {};
\node [draw, shape=circle, scale=.8] (d) at  (11.25, -3) {};

\draw[thick, dotted](-0.5,-0.6) rectangle (0.5,0.7);
\draw[thick, dotted](1,-0.6) rectangle (3.5,0.7);
\draw[thick, dotted](4,-0.6) rectangle (8,0.7);
\draw[thick, dotted](8.5,-0.6) rectangle (14,0.7);

\node [scale=1] at (0.1,0.3) {$w_{1,1}$};
\node [scale=1] at (1.6,0.3) {$w_{2,1}$};
\node [scale=1] at (3.1,0.3) {$w_{2,2}$};
\node [scale=1] at (4.6,0.3) {$w_{3,1}$};
\node [scale=1] at (6.1,0.3) {$w_{3,2}$};
\node [scale=1] at (7.6,0.3) {$w_{3,3}$};
\node [scale=1] at (9.1,0.3) {$w_{4,1}$};
\node [scale=1] at (10.6,0.3) {$w_{4,2}$};
\node [scale=1] at (12.1,0.3) {$w_{4,3}$};
\node [scale=1] at (13.6,0.3) {$w_{4,4}$};

\node [scale=1.1] at (0,-3.4) {$u_{1}$};
\node [scale=1.1] at (2.25,-3.4) {$u_{2}$};
\node [scale=1.1] at (6,-3.4) {$u_{3}$};
\node [scale=1.1] at (11.25,-3.4) {$u_{4}$};

\node [scale=1.1] at (0,1.15) {$V_1$};
\node [scale=1.1] at (2.25,1.15) {$V_2$};
\node [scale=1.1] at (6,1.15) {$V_3$};
\node [scale=1.1] at (11.25,1.15) {$V_4$};
\node [scale=1.1] at (15,1.6) {$H_4=K_{10}$};
\node [scale=1.5] at (-3.5,-1) {$G_4$:};

\draw[thick,dashed] (6.75,0.3) ellipse (9.3cm and 1.7cm);

\draw(0)--(a)--(1);\draw(3)--(a)--(6);\draw(1)--(b)--(2);\draw(4)--(b)--(7);\draw(3)--(c)--(4);\draw(5)--(c)--(8);\draw(6)--(d)--(7);\draw(8)--(d)--(9);
\end{tikzpicture}
\caption{\small \cite{broadcast} Graphs $G_m$ and $H_m$ with $H_m \subset G_m$ such that $\frac{\dim(H_m)}{\dim(G_m)}$ can be arbitrarily large.}\label{fig_subgraph}
\end{figure}
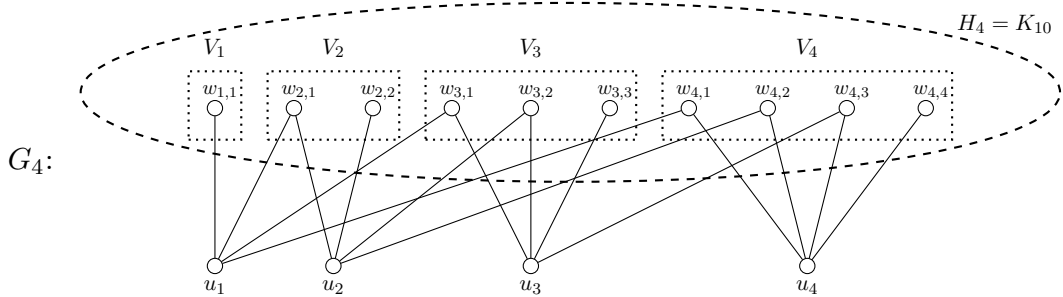

\begin{corollary}
There exist connected graphs $G$ and $H$ such that $H \subset G$ and $\frac{\dim_f(H)}{\dim_f(G)}$ can be arbitrarily large.
\end{corollary}

Next, we show that edge dimension is not a monotone parameter on subgraph inclusion.

\begin{lemma}\label{edim_subgraph}
There exist connected graphs $G$ and $H$ such that $H \subset G$ and $\edim(G) <\edim(H)$. 
\end{lemma}

\begin{proof}
Let $H=K_{4,2}$ and $G$ be the graphs in Figure~\ref{fig_subgraph_edim}; notice that $H \subset G$. Then $\edim(H)=\edim(K_{4,2})=4$ (see~\cite{edim}) and $\edim(G) \le 3$ since $\{a,b,c\}$ forms an edge resolving set for $G$; thus, $\edim(G) <\edim(H)$.
\end{proof}

\begin{figure}[ht]
\centering
\begin{tikzpicture}[scale=.7, transform shape]

\node [draw, shape=circle, scale=.8] (1) at  (0.2,6) {};
\node [draw, shape=circle, scale=.8] (2) at  (0.2,4) {};
\node [draw, shape=circle, scale=.8] (3) at  (0.2,2) {};
\node [draw, shape=circle, scale=.8] (4) at  (0.2,0) {};
\node [draw, shape=circle, scale=.8] (a) at  (5, 4) {};
\node [draw, shape=circle, scale=.8] (b) at  (5, 2) {};

\node [draw, fill=black, shape=circle, scale=.8] (0) at  (-0.5, 1) {};
\node [draw, shape=circle, scale=.8] (00) at  (-0.5, -1.5) {};
\node [draw, fill=black, shape=circle, scale=.8] (01) at  (3.5, -1.5) {};
\node [draw, fill=black, shape=circle, scale=.8] (02) at  (-2, 1) {};
\node [draw, shape=circle, scale=.8] (03) at  (-1.5, 3) {};

\draw[thick, dotted](-0.1,-0.25) rectangle (5.3,6.3);
\node [scale=1.1] at (2.65,6.7) {$H=K_{4,2}$};
\node [scale=1.5] at (-4.5,2.5) {$G$:};
\node [scale=1.2] at (-2.3,1) {$a$};
\node [scale=1.2] at (-0.8,1) {$b$};
\node [scale=1.2] at (3.85,-1.5) {$c$};
\node [scale=1] at (-1.4,3.65) {$(1,3,2)$};
\node [scale=1] at (-2.4,2.2) {$(0,3,2)$};
\node [scale=1] at (-2,-0.3) {$(0,2,1)$};
\node [scale=1] at (1.35,-1.8) {$(1,2,0)$};
\node [scale=1] at (0.55,-0.75) {$(1,1,1)$};
\node [scale=1] at (4.55,-0.7) {$(2,2,0)$};
\node [scale=1] at (-0.75,1.6) {$(3,0,2)$};
\node [scale=1] at (-0.75,0.35) {$(2,0,2)$};
\node [scale=1] at (2.1,5.7) {$(3,2,2)$};
\node [scale=1] at (0.65,5) {$(3,2,1)$};
\node [scale=1] at (1,4.2) {$(2,2,2)$};
\node [scale=1] at (0.75,3.35) {$(2,2,1)$};
\node [scale=1] at (0.75,2.65) {$(3,1,2)$};
\node [scale=1] at (1,1.75) {$(3,1,1)$};
\node [scale=1] at (0.75,1.05) {$(2,1,2)$};
\node [scale=1] at (2,0.35) {$(2,1,1)$};

\draw(a)--(1)--(b);\draw(a)--(2)--(b);\draw(a)--(3)--(b);\draw(a)--(4)--(b);\draw(3)--(0)--(4);\draw(2)--(03)--(02)--(00)--(01)--(b);\draw(00)--(4);

\end{tikzpicture}
\caption{\small Graphs $G$ and $H$ such that $H\subset G$ and $\edim(G) <\edim(H)$, where the 3-vector next to each edge $e\in E(G)$ is $\code_{\{a,b,c\}}(e)$.}\label{fig_subgraph_edim}
\end{figure}
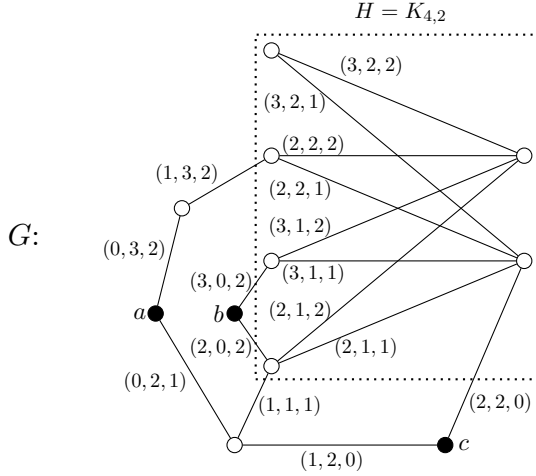

Based on the construction of graphs in Figure~\ref{fig_subgraph_edim}, we show the existence of graphs $G$ and $H$ such that $H \subset G$ and both $\edim(H)-\edim(G)$ and $\edim_f(H)-\edim_f(G)$ can be arbitrarily large.

\begin{proposition}
There exist graphs $G$ and $H$ such that $H \subset G$ and both $\edim(H)-\edim(G)$ and $\edim_f(H)-\edim_f(G)$ can be arbitrarily large.
\end{proposition}

\begin{proof}
Let $k\ge2$. Let $H_k$ be a graph with vertex set $V(H_k)=\cup_{i=1}^{3k}\{x_i, y_i\}$ and edge set $E(H_k)=\cup_{i=1}^{3k-1}\{x_ix_{i+1}, y_iy_{i+1}, x_iy_{i+1}, y_ix_{i+1}\}$. We note that $x_i$ and $y_i$ are twins in $H_k$, where $i\in\{1,2,\ldots, 3k\}$. Let $G_k$ be the graph obtained from the disjoint union of $H_k$, $K_1$, and $k$ copies of $P_4$ that is given by $a_i, b_i, c_i, d_i$ for $i\in\{1,2,\ldots, k\}$ as follows: (i) the vertex $z\in V(K_1)$ is adjacent to $x_1$ and $y_1$ in $G_k$; (ii) $y_{3j+1}$ is adjacent to $b_{j+1}$, $y_{3j+2}$ is adjacent to $a_{j+1}$, and $y_{3j+3}$ is adjacent to $d_{j+1}$ in $G_k$, where $j\in\{0,1, \ldots, k-1\}$. See Figure~\ref{fig_subgraph_edim_large} for graphs $G_3$ and $H_3$.

First, we show that $\edim(H_k)=3k=\edim_f(H_k)$. Note that $\edim(H_k)\ge 3k$ by Observation~\ref{obs_twin}(b) and $\edim_f(H_k)\ge 3k$ by Observation~\ref{obs_twin}(d). Since $\cup_{i=1}^{3k}\{x_i\}$ forms an edge resolving set of $H_k$, $\edim_f(H_k) \le \edim(H_k) \le 3k$ by observation~\ref{obs_frac}(b). 

Second, we show that $\edim_f(G_k) \le\edim(G_k) \le 1+2k$. Since $\{z\} \cup(\cup_{i=1}^{k} \{a_i, c_i\})$ forms an edge resolving set of $G_k$, $\edim_f(G_k) \le\edim(G_k) \le 2k+1$ by Observation~\ref{obs_frac}(b).

Thus, $\edim(H_k)-\edim(G_k) \ge 3k-(1+2k)=k-1 \rightarrow \infty$  and $\edim_f(H_k)-\edim_f(G_k) \ge k-1 \rightarrow \infty$ as $k \rightarrow \infty$.~\hfill
\end{proof}

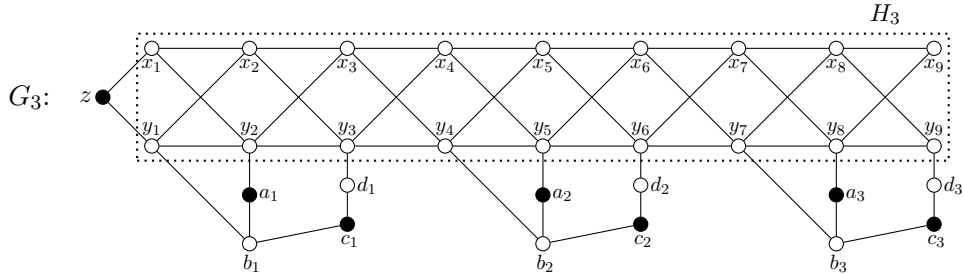
\begin{figure}[ht]
\centering
\begin{tikzpicture}[scale=.65, transform shape]

\node [draw, fill=black, shape=circle, scale=.8] (0) at  (-1,1) {};

\node [draw, shape=circle, scale=.8] (11) at  (0,2) {};
\node [draw, shape=circle, scale=.8] (12) at  (0,0) {};
\node [draw, shape=circle, scale=.8] (21) at  (2,2) {};
\node [draw, shape=circle, scale=.8] (22) at  (2,0) {};
\node [draw, shape=circle, scale=.8] (31) at  (4,2) {};
\node [draw, shape=circle, scale=.8] (32) at  (4,0) {};

\node [draw, shape=circle, scale=.8] (a) at  (2,-2) {};
\node [draw, fill=black, shape=circle, scale=.8] (a1) at  (2,-1) {};
\node [draw, shape=circle, scale=.8] (a21) at  (4,-0.8) {};
\node [draw, fill=black, shape=circle, scale=.8] (a22) at  (4,-1.6) {};

\node [draw, shape=circle, scale=.8] (41) at  (6,2) {};
\node [draw, shape=circle, scale=.8] (42) at  (6,0) {};
\node [draw, shape=circle, scale=.8] (51) at  (8,2) {};
\node [draw, shape=circle, scale=.8] (52) at  (8,0) {};
\node [draw, shape=circle, scale=.8] (61) at  (10,2) {};
\node [draw, shape=circle, scale=.8] (62) at  (10,0) {};

\node [draw, shape=circle, scale=.8] (b) at  (8,-2) {};
\node [draw, fill=black, shape=circle, scale=.8] (b1) at  (8,-1) {};
\node [draw, shape=circle, scale=.8] (b21) at  (10,-0.8) {};
\node [draw, fill=black, shape=circle, scale=.8] (b22) at  (10,-1.6) {};

\node [draw, shape=circle, scale=.8] (71) at  (12,2) {};
\node [draw, shape=circle, scale=.8] (72) at  (12,0) {};
\node [draw, shape=circle, scale=.8] (81) at  (14,2) {};
\node [draw, shape=circle, scale=.8] (82) at  (14,0) {};
\node [draw, shape=circle, scale=.8] (91) at  (16,2) {};
\node [draw, shape=circle, scale=.8] (92) at  (16,0) {};

\node [draw, shape=circle, scale=.8] (c) at  (14,-2) {};
\node [draw, fill=black, shape=circle, scale=.8] (c1) at  (14,-1) {};
\node [draw, shape=circle, scale=.8] (c21) at  (16,-0.8) {};
\node [draw, fill=black, shape=circle, scale=.8] (c22) at  (16,-1.6) {};

\draw[thick, dotted](-0.3,-0.3) rectangle (16.3,2.3);
\node [scale=1.3] at (15,2.7) {$H_3$};
\node [scale=1.5] at (-2.5,1) {$G_3$:};
\draw(11)--(0)--(12);
\draw(21)--(11)--(22);\draw(31)--(21)--(32);\draw(41)--(31)--(42);\draw(51)--(41)--(52);\draw(61)--(51)--(62);\draw(71)--(61)--(72);\draw(81)--(71)--(82);\draw(91)--(81)--(92);
\draw(21)--(12)--(22);\draw(31)--(22)--(32);\draw(41)--(32)--(42);\draw(51)--(42)--(52);\draw(61)--(52)--(62);\draw(71)--(62)--(72);\draw(81)--(72)--(82);\draw(91)--(82)--(92);
\draw(12)--(a);\draw(22)--(a1)--(a)--(a22)--(a21)--(32);
\draw(42)--(b);\draw(52)--(b1)--(b)--(b22)--(b21)--(62);
\draw(72)--(c);\draw(82)--(c1)--(c)--(c22)--(c21)--(92);

\node [scale=1.3] at (-1.35,1) {$z$};

\node [scale=1.1] at (0,1.65) {$x_1$};
\node [scale=1.1] at (2,1.65) {$x_2$};
\node [scale=1.1] at (4,1.65) {$x_3$};
\node [scale=1.1] at (6,1.65) {$x_4$};
\node [scale=1.1] at (8,1.65) {$x_5$};
\node [scale=1.1] at (10,1.65) {$x_6$};
\node [scale=1.1] at (12,1.65) {$x_7$};
\node [scale=1.1] at (14,1.65) {$x_8$};
\node [scale=1.1] at (16,1.65) {$x_9$};
\node [scale=1.1] at (0,0.35) {$y_1$};
\node [scale=1.1] at (2,0.35) {$y_2$};
\node [scale=1.1] at (4,0.35) {$y_3$};
\node [scale=1.1] at (6,0.35) {$y_4$};
\node [scale=1.1] at (8,0.35) {$y_5$};
\node [scale=1.1] at (10,0.35) {$y_6$};
\node [scale=1.1] at (12,0.35) {$y_7$};
\node [scale=1.1] at (14,0.35) {$y_8$};
\node [scale=1.1] at (16,0.35) {$y_9$};

\node [scale=1.1] at (2.4,-1) {$a_1$};
\node [scale=1.1] at (2.05,-2.4) {$b_1$};
\node [scale=1.1] at (4.05,-1.95) {$c_1$};
\node [scale=1.1] at (4.4,-0.8) {$d_1$};
\node [scale=1.1] at (8.4,-1) {$a_2$};
\node [scale=1.1] at (8.05,-2.4) {$b_2$};
\node [scale=1.1] at (10.05,-1.95) {$c_2$};
\node [scale=1.1] at (10.4,-0.8) {$d_2$};
\node [scale=1.1] at (14.4,-1) {$a_3$};
\node [scale=1.1] at (14.05,-2.4) {$b_3$};
\node [scale=1.1] at (16.05,-1.95) {$c_3$};
\node [scale=1.1] at (16.4,-0.8) {$d_3$};

\end{tikzpicture}
\caption{\small Graphs $G_k$ and $H_k$ such that $H_k\subset G_k$ and both $\edim(H_k)-\edim(G_k)$ and $\edim_f(H_k)-\edim_f(G_k)$ can be arbitrarily large, where $k\ge2$.}\label{fig_subgraph_edim_large}
\end{figure}

\begin{question}
Are there connected graphs $G$ and $H$ with $H \subset G$ such that both $\frac{\edim(H)}{\edim(G)}$ and $\frac{\edim_f(H)}{\edim_f(G)}$ can be arbitrarily large?
\end{question}

We conclude this section with some general results on fractional edge dimension. For any connected graph $G$ of order at least three, it is easy to see that $1\le \edim_f(G) \le \frac{|V(G)|}{2}$. We characterize graphs $G$ satisfying $\edim_f(G)=1$. We show that $\dim_f(G)=\frac{|V(G)|}{2}$ implies $\edim_f(G)=\frac{|V(G)|}{2}$, while there exist graphs $G$ with $\edim_f(G)=\frac{|V(G)|}{2}>\dim_f(G)$; in fact, there exists a family of graphs $G$ with $\edim_f(G)=\frac{|V(G)|}{2}$ such that $\frac{\edim_f(G)}{\dim_f(G)}$ can be arbitrarily large.

\begin{proposition}\label{edim_frac_bounds}
If $G$ is a connected graph of order $n\ge 3$, then $1\le \edim_f(G)\le \frac{n}{2}$.
\end{proposition}

\begin{proof}
Let $G$ be a connected graph of order $n\ge 3$. The lower bound follows from the definition of fractional edge dimension. Next, we prove the upper bound. Let $g: V(G) \rightarrow [0,1]$ be a function defined by $g(v)=\frac{1}{2}$ for each $v\in V(G)$; then $g(V(G))=\frac{n}{2}$. It suffices to show that $g$ is an edge resolving function of $G$. Let $e_1=ab$ and $e_2=cd$ be any distinct edges in $G$. If $\{a,b\} \cap \{c,d\} =\emptyset$, then $R_e\{e_1, e_2\} \supseteq \{a,b,c,d\}$. If $\{a,b\} \cap \{c,d\} \neq \emptyset$, say $a=c$, by relabeling the vertices of $G$ if necessary, then $b\neq d$ and $R_e\{e_1, e_2\} \supseteq \{b,d\}$. In each case, $|R_e\{e_1, e_2\}|\ge 2$ and $g(R_e\{e_1, e_2\}) \ge 2 (\frac{1}{2})=1$. So, $g$ is an edge resolving function of $G$, and thus $\edim_f(G) \le g(V(G))=\frac{n}{2}$.~\hfill
\end{proof}

Next, we characterize graphs achieving the lower bound of Proposition~\ref{edim_frac_bounds}. We recall the following result.

\begin{proposition}\emph{\cite{frac_sdim}}\label{dim_frac_1} 
For any connected graph $G$ of order $n \ge 2$, $\dim_f(G)=1$ if and only if $G=P_n$.
\end{proposition}

In proving Proposition~\ref{edim_frac_1}, we use the same technique used for Proposition~\ref{dim_frac_1} in~\cite{frac_sdim} by adjusting to edge resolving function.  

\begin{proposition}\label{edim_frac_1}
For any connected graph $G$ of order $n \ge 3$, $\edim_f(G)=1$ if and only if $G=P_n$.
\end{proposition}

\begin{proof}
$(\Leftarrow)$ Let $P_n$ be given by $u_1, u_2, \ldots, u_n$, where $n\ge 3$, and let $g: V(P_n) \rightarrow [0,1]$ be a function defined by $g(u_1)=\frac{1}{2}=g(u_n)$ and $g(v)=0$ for each $v\in V(P_n)-\{u_1, u_n\}$. For any distinct edges $e_1, e_2 \in E(P_n)$, $R_e\{e_1, e_2\} \supseteq \{u_1, u_n\}$; thus, $g(R_e\{e_1, e_2\}) \ge g(u_1)+g(u_n)=1$. So, $g$ is an edge resolving function of $P_n$ with $g(V(P_n))=1$, and thus $\edim_f(P_n) \le 1$. By Proposition~\ref{edim_frac_bounds}, $\edim_f(P_n)=1$.

($\Rightarrow$) Let $\mathcal{C}=\displaystyle\bigcap_{e_1, e_2 \in E(G)} R_e\{e_1, e_2\}$, where the intersection is taken over all pairs of distinct edges of $G$; notice that $\mathcal{C} \neq \emptyset$ if and only if $G=P_n$ for some $n$. Let $\edim_f(G)=1$. Assume, to the contrary, that $G\neq P_n$ for any $n$. Let $e_1, e_2, e_3,e_4\in E(G)$ such that $e_1 \neq e_2$ and $e_3\neq e_4$. Let $A=R_e\{e_1, e_2\}$, $B=R_e\{e_3, e_4\}$, $A'=A-(A\cap B)$, $B'=B-(A\cap B)$ and let $g:V(G)\rightarrow [0,1]$ be an edge resolving function of $G$ with $g(V(G))=1$. Since $g$ is an edge resolving function of $G$, we have the following: 
\begin{equation}\label{eq1}
g(A')+g(A\cap B) \ge 1 \mbox{ and } g(B')+g(A\cap B)\ge 1 \tag{$\star$}.
\end{equation}
By the assumption that $\edim_f(G)=1$, we have 
\begin{equation}\label{eq2}
g(A')+g(A\cap B)+g(B')=1 \tag{$\dagger$}.
\end{equation}  
Now, from~(\ref{eq1}) and~(\ref{eq2}), $g(A')=0=g(B')$. Since $e_1, e_2, e_3, e_4$ are arbitrary, $g$ is zero except on $\mathcal{C}$. Since $G\neq P_n$, $\mathcal{C}=\emptyset$ and $g(V(G))=0$; this contradicts the assumption that $g(V(G))=1$.~\hfill
\end{proof}

Next, we show that, for any connected graph $G$ of order at least three, $\dim_f(G)=\frac{|V(G)|}{2}$ implies $\edim_f(G)=\frac{|V(G)|}{2}$, but not vice versa. 

\begin{theorem}\label{edim_frac_n2}
Let $G$ be a connected graph of order $n\ge 3$.
\begin{itemize}
\item[(a)] If $\dim_f(G)=\frac{n}{2}$, then $\edim_f(G)=\frac{n}{2}$.
\item[(b)] There exists a graph $G$ satisfying $\edim_f(G)=\frac{n}{2}>\dim_f(G)$. Moreover, there exists a family of graphs $G$ with $\edim_f(G)=\frac{n}{2}>\dim_f(G)$ such that $\frac{\edim_f(G)}{\dim_f(G)}$ can be arbitrarily large.
\end{itemize}
\end{theorem}

\begin{proof}
Let $G$ be a connected graph of order $n\ge 3$.

\medskip

(a) Suppose $\dim_f(G)=\frac{n}{2}$. By Theorem~\ref{n_2}, $G\in H[\mathcal{K} \cup \overline{\mathcal{K}}]$ for some connected graph $H$. Let $g: V(G)\rightarrow[0,1]$ be any edge resolving function of $G$; we show that $g(V(G))\ge \frac{n}{2}$. First, let $G=K_n$ with $V(G)=\{u_1, u_2, \ldots, u_n\}$. By Observation~\ref{obs_twin}(d), $g(u_i)+g(u_j)\ge1$ for any distinct $i,j\in\{1,2,\ldots, n\}$. By summing over the ${n \choose 2}$ inequalities, we have $(n-1) \sum_{i=1}^{n} g(u_i) \ge {n \choose 2}$, and thus $g(V(G))=\sum_{i=1}^{n}g(u_i)\ge \frac{n}{2}$. Second, let $G\neq K_n$, and let $\mathcal{C}=\{Q_1, Q_2, \ldots, Q_t\}$ be the collection of twin equivalence classes of $G$, where $t\ge2$; then, by Theorem~\ref{n_2}, $n=\sum_{i=1}^{t}|Q_i|$ and $|Q_i|\ge 2$ for each $i\in\{1,2,\ldots, t\}$. Again, by Observation~\ref{obs_twin}(d), $g(Q_i) \ge \frac{|Q_i|}{2}$ for each $i\in \{1,2, \ldots, t\}$. By summing over the $t$ inequalities, we have $g(V(G))=\sum_{i=1}^{t}g(Q_i)\ge \sum_{i=1}^{t} \frac{|Q_i|}{2}=\frac{n}{2}$. So, in each case, $g(V(G))\ge \frac{n}{2}$, and thus $\edim_f(G)\ge \frac{n}{2}$. Since $\edim_f(G)\le \frac{n}{2}$ by Proposition~\ref{edim_frac_bounds}, $\edim_f(G)=\frac{n}{2}$.

\medskip 

(b) If $G=K_{a_1, a_2, \ldots, a_k}$ is a complete $k$-partite graph of order $n=\sum_{i=1}^{k}a_i\ge5$ such that $k\ge 3$ and $a_i=1$ for exactly one $i\in\{1,2,\ldots, k\}$, then $\edim_f(G)=\frac{n}{2}>\frac{n-1}{2}=\dim_f(G)$ (see Theorem~\ref{dim_frac_graph}(e) and Proposition~\ref{edim_frac_kpartite}). 

\medskip

Next, we show the existence of graphs $G$ such that $\edim_f(G)=\frac{n}{2}$ and $\frac{\edim_f(G)}{\dim_f(G)}$ can be arbitrarily large. We recall the family of graphs $G$ constructed in~\cite{nina} to show that $\frac{\edim(G)}{\dim(G)}$ can be arbitrarily large. For each integer $k\ge 3$ and the set $X=\{0,1,\ldots,k-1\}$, let $G_k$ be a graph of order $k+2^k$ such that $V(G_k)=A \cup B$, where $|A|=k$ and $|B|=2^k$. Let $A=\{a_0, a_1, \ldots, a_{k-1}\}$ and $B=\{b_S:S \subseteq X\}$. The edge set is specified as follows: (i) each of the sets $A$ and $B$ induces a clique in $G_k$; (ii) indexing the elements of $B$ by subsets of $X$, let $b_S\in B$ be adjacent to $a_i\in A$ if $i\in S$ for $S \subseteq X$; (iii) there are no other edges. See Figure~\ref{fig_edim_big} for the graph $G_3$. Let $k\ge 3$ be an integer.

First, we show that $\dim_f(G_k) =k$. Let $g:V(G_k) \rightarrow [0,1]$ be any resolving function of $G_k$. Note that, for each $i\in\{0,1,\ldots, k-1\}$, $R_v\{b_{\{i\}}, b_{\{i, i+1\}}\}=\{b_{\{i\}}, b_{\{i, i+1\}}, a_{i+1}\}$ and $g(b_{\{i\}})+g(b_{\{i,i+1\}})+g(a_{i+1})\ge 1$, where the subscript is taken modulo $k$. By summing over the $k$ inequalities, we have $g(A)+\sum_{i=0}^{k-1}(g(b_{\{i\}})+g(b_{\{i, i+1\}}))\ge k$, where the subscript is taken modulo $k$. So, $g(V(G_k)) \ge k$, and hence $\dim_f(G_k)\ge k$. On the other hand, $\dim_f(G_k) \le \dim(G_k)=k$ by Observation~\ref{obs_frac}(a) and the fact that $\dim(G_k)=k$ as shown in~\cite{nina}.

Second, we show that $\edim_f(G_k)=\frac{k+2^k}{2}$. Since $\edim_f(G_k) \le \frac{k+2^k}{2}=\frac{|V(G_k)|}{2}$ by Poposition~\ref{edim_frac_bounds}, it suffices to show that $\edim_f(G_k) \ge \frac{k+2^k}{2}$. Let $h:V(G_k) \rightarrow [0,1]$ be any edge resolving function of $G_k$. Note that, for each $i\in\{0,1,\ldots, k-1\}$, $R_e\{b_{\{i, i+1\}}a_i, b_{\{i, i+1\}}a_{i+1}\}=\{a_i, a_{i+1}\}$ and $h(a_i)+h(a_{i+1})\ge 1$, where the subscript is taken modulo $k$. By summing over the $k$ inequalities, we have $2h(A)\ge k$, i.e., $h(A) \ge \frac{k}{2}$ ($\blacktriangle$). Since $R_e\{a_0b_{\{0\}}, a_0b_X\}=\{b_{\{0\}}, b_X\}$, we have $h(b_{\{0\}})+h(b_X)\ge 1$ ($\blacktriangleleft$). We also note that, for any distinct $Y,Z \subsetneq X$ with $Y\neq\{0\}$ and $Z \neq\{0\}$, $R_e\{b_Xb_Y, b_X b_Z\}=\{b_Y, b_Z\}$ and $h(b_Y)+h(b_Z)\ge 1$. By summing over the ${2^k-2 \choose 2}$ inequalities, we have $(2^k-3)[(\sum_{S\subseteq X} h(b_S))-(h(b_{\{0\}})+h(b_X))] \ge {2^k-2 \choose 2}$, i.e., $(\sum_{S\subseteq X} h(b_S))-(h(b_{\{0\}})+h(b_X)) \ge \frac{2^{k}-2}{2}$ ($\blacktriangleright$). By summing over the three inequalities ($\blacktriangle$), ($\blacktriangleleft$) and ($\blacktriangleright$), we have $h(V(G_k))=h(A)+h(B) \ge\frac{k}{2}+\frac{2^k}{2}$; thus, $\edim_f(G_k)\ge \frac{k+2^k}{2}$. 

Therefore, $\edim_f(G_k)=\frac{|V(G_k)|}{2}$ and $\frac{\edim_f(G_k)}{\dim_f(G_k)}=\frac{k+2^k}{2k} \rightarrow \infty$ as $k \rightarrow \infty$.~\hfill
\end{proof}

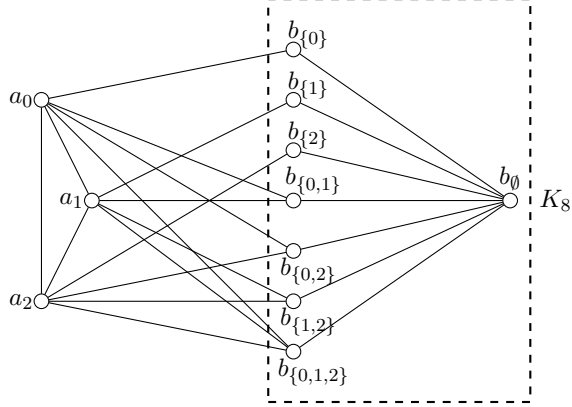
\begin{figure}[ht]
\centering
\begin{tikzpicture}[scale=.67, transform shape]

\node [draw, shape=circle, scale=.8] (a1) at  (-1, 5) {};
\node [draw, shape=circle, scale=.8] (a2) at  (0, 3) {};
\node [draw, shape=circle, scale=.8] (a3) at  (-1, 1) {};

\node [draw, shape=circle, scale=.8] (b1) at  (4, 6) {};
\node [draw, shape=circle, scale=.8] (b2) at  (4, 5) {};
\node [draw, shape=circle, scale=.8] (b3) at  (4, 4) {};
\node [draw, shape=circle, scale=.8] (b4) at  (4, 3) {};
\node [draw, shape=circle, scale=.8] (b5) at  (4, 2) {};
\node [draw, shape=circle, scale=.8] (b6) at  (4, 1) {};
\node [draw, shape=circle, scale=.8] (b7) at  (4, 0) {};

\node [draw, shape=circle, scale=.8] (c) at  (8.3, 3) {};

\node [scale=1.3] at (-1.4,5) {$a_0$};
\node [scale=1.3] at (-0.4,3) {$a_1$};
\node [scale=1.3] at (-1.4,1) {$a_2$};

\node [scale=1.3] at (4.3, 6.3) {$b_{\{0\}}$};
\node [scale=1.3] at (4.3, 5.3) {$b_{\{1\}}$};
\node [scale=1.3] at (4.3,4.3) {$b_{\{2\}}$};
\node [scale=1.3] at (4.4, 3.4) {$b_{\{0,1\}}$};
\node [scale=1.3] at (4.3, 1.6) {$b_{\{0,2\}}$};
\node [scale=1.3] at (4.3,0.6) {$b_{\{1,2\}}$};
\node [scale=1.3] at (4.4,-0.4) {$b_{\{0,1,2\}}$};
\node [scale=1.3] at (8.3,3.45) {$b_{\emptyset}$};

\node [scale=1.3] at (9.2,3) {$K_8$};

\draw(a1)--(a2)--(a3)--(a1);
\draw(b1)--(c);\draw(b2)--(c);\draw(b3)--(c);\draw(b4)--(c);\draw(b5)--(c);\draw(b6)--(c);\draw(b7)--(c);
\draw(b1)--(a1)--(b4);\draw(b5)--(a1)--(b7);\draw(b2)--(a2);\draw(a2)--(b4);\draw(b6)--(a2)--(b7);\draw(b3)--(a3)--(b5);\draw(b6)--(a3)--(b7);

\draw[thick, dashed] (3.5,-1) rectangle (8.7,7); 

\end{tikzpicture}
\caption{\small The graph $G_3$.}\label{fig_edim_big}
\end{figure}
 
\begin{question}
Is there a connected graph $G$ such that $\frac{\dim_f(G)}{\edim_f(G)}$ can be arbitrarily large?
\end{question}

\begin{question}
Can we characterize connected graphs $G$ for which $\edim_f(G)=\frac{|V(G)|}{2}$ hold?
\end{question}


\section{The fractional edge dimension of some graph classes}\label{sec_graphs}

In this section, we determine $\edim_f(G)$ when $G$ is a tree, a cycle, the Petersen graph, a wheel graph, a complete multi-partite graph and a grid graph (also known as the Cartesian product of two paths), respectively. 

We recall some terminology. Fix a tree $T$. A leaf $\ell$ is called a \emph{terminal vertex} of a major vertex $v$ if $d(\ell,v) < d(\ell,w)$ for every other major vertex $w$ in $T$. The terminal degree, $ter(v)$, of a major vertex $v$ is the number of terminal vertices of $v$ in $T$, and an \emph{exterior major vertex} is a major vertex that has positive terminal degree. Let $M(T)$ be the set of exterior major vertices of $T$. Let $M_1(T)=\{w\in M(T): ter(w)=1\}$ and $M_2(T)=\{w\in M(T): ter(w) \ge 2\}$; then $M(T)=M_1(T) \cup M_2(T)$. Let $\sigma(T)$ denote the number of leaves of $T$, $ex(T)$ the number of exterior major vertices of $T$, and let $ex_1(T)=|M_1(T)|$. For each $v\in M(T)$, let $T_v$ be the subtree of $T$ induced by $v$ and all vertices belonging to the paths joining $v$ with its terminal vertices, and let $L(v)$ be the set of terminal vertices of $v$ in $T$.

First, we recall the fractional metric dimension of some graph classes.

\begin{theorem}\label{dim_frac_graph} 
\begin{itemize}
\item[(a)] \emph{\cite{yi}} For any non-trivial tree $T$, $\dim_f (T)=\frac{1}{2}(\sigma(T)-ex_1(T))$.
\item[(b)] \emph{\cite{fdim}} For $n\ge 3$, $\dim_f(C_n)=\left\{
\begin{array}{ll}
\frac{n}{n-1} & \mbox{ if $n$ is odd},\\
\frac{n}{n-2} & \mbox{ if $n$ is even}. 
\end{array}\right.$
\item[(c)] \emph{\cite{fdim}} For the Petersen graph $\mathcal{P}$, $\dim_f(\mathcal{P})=\frac{5}{3}$.
\item[(d)] \emph{\cite{fdim}} For the wheel graph $W_n=K_1+C_{n-1}$ of order $n\ge 4$, $\dim_f(W_n)=\left\{
\begin{array}{ll}
2 & \mbox{ if } n\in\{4,5\},\\
\frac{3}{2} & \mbox{ if } n=6,\\
\frac{n-1}{4} & \mbox{ if } n \ge 7. 
\end{array}\right.$
\item[(e)] \emph{\cite{yi}} For $k\ge2$, let $G=K_{a_1, a_2, \ldots, a_k}$ be a complete $k$-partite graph of order $n=\sum_{i=1}^{k}a_i$, and let $s$ be the number of partite sets of $G$ consisting of exactly one element. Then 
\begin{equation*}
\dim_f(G)=\left\{
\begin{array}{ll}
\frac{n-1}{2} & \mbox{ if } s=1,\\
\frac{n}{2} & \mbox{ otherwise}. 
\end{array}\right.
\end{equation*}
\item[(f)] \emph{\cite{fdim}} For $s,t\ge 2$, $\dim_f(P_s \square P_t)=2$, where $P_s \square P_t$ denotes the Cartesian product of $P_s$ and $P_t$.
\end{itemize}
\end{theorem}

Next, we determine $\edim_f(T)$ for a tree $T$; we show that $\edim_f(T)=\dim_f(T)$. We begin with the following useful lemma, which can be obtained from Lemma 18 of~\cite{frac_kdim} by adjusting the statement and its proof for two distinct edges and edge resolving sets. 

\begin{lemma}\label{lemma_tree2}
Let $T$ be a tree with $ex(T) \ge 2$. For $w\in M_2(T)$, let $e_1\in E(T_w)$ and $e_2\in E(T)-E(T_w)$. Then either $R_e\{e_1,e_2\} \supseteq V(T_w)$ or $R_e\{e_1,e_2\} \supseteq V(T_{w'})$ for some $w'\in M_2(T)-\{w\}$. 
\end{lemma}

\begin{proposition}\label{edim_frac_tree}
For any tree $T$ of order at least three, $\edim_f(T)=\frac{1}{2}(\sigma(T)-ex_1(T))$.
\end{proposition}

\begin{proof}
Let $T$ be a tree of order at least three. If $ex(T)=0$, then $T$ is a path and $\edim_f(T)=1=\frac{1}{2}(\sigma(T)-ex_1(T))$ by Proposition~\ref{edim_frac_1}. So, suppose $ex(T)\ge 1$; then $M_2(T) \neq \emptyset$. 

\medskip

First, we show that $\edim_f(T) \ge \frac{1}{2}(\sigma(T)-ex_1(T))$. Let $g:V(T)\rightarrow [0,1]$ be any edge resolving function of $T$. Fix $v\in M_2(T)$ with $ter(v)=a \ge2$. Let $N(v)=\{s_1, s_2, \ldots, s_a\}$ and $L(v)=\{\ell_1, \ell_2, \ldots, \ell_a\}$ such that $s_i$ lies on the $v-\ell_i$ path, where $i\in\{1,2,\ldots, a\}$. For each $i\in\{1,2,\ldots, a\}$, let $P^i$ denote the $s_i-\ell_i$ path. Note that, for any distinct $i,j\in\{1,2,\ldots, a\}$, $R_e\{vs_i, vs_j\} =V(P^i) \cup V(P^j)$; thus, $g(V(P^i))+g(V(P^j)) \ge 1$. By summing over the ${a \choose 2}$ inequalities, we have $(a-1) \sum_{i=1}^{a}g(V(P^i))\ge {a \choose 2}$, which implies $g(V(T_v)) \ge \sum_{i=1}^{a}g(V(P^i)) \ge \frac{a}{2}$. So, 
$$g(V(T)) \ge \!\!\!\sum_{v\in M_2(T)} \!\!\!g(V(T_v)) \ge\! \!\!\sum_{v\in M_2(T)} \!\!\!\frac{ter(v)}{2}=\frac{1}{2}\left(\sum_{v\in M(T)} \!\!\!ter(v)-\!\!\!\!\!\sum_{w\in M_1(T)} \!\!\!ter(w)\right)=\frac{1}{2}(\sigma(T)-ex_1(T)),$$
and thus $\edim_f(T) \ge \frac{1}{2}(\sigma(T)-ex_1(T))$. 

\medskip

Next, we show that $\edim_f(T) \le \frac{1}{2}(\sigma(T)-ex_1(T))$. For $x\in V(T)$, let $h:V(T) \rightarrow [0,1]$ be a function defined by
\begin{equation*}
h(x)=\left\{
\begin{array}{ll}
\frac{1}{2} & \mbox{ if $x$ is a terminal vertex of an exterior major vertex $w\in M_2(T)$},\\
0 & \mbox{ otherwise.}
\end{array}\right.
\end{equation*}
Notice that $h(V(T))=\frac{1}{2}(\sigma(T)-ex_1(T))$. It suffices to show that $h$ is an edge resolving function of $T$. Let $e_1, e_2 \in E(T)$ with $e_1\neq e_2$. We consider three cases: (1) $e_1,e_2\in E(T_w)$ for some $w\in M_2(T)$; (2) $e_1\in E(T_w)$ and $e_2\not\in E(T_w)$ for some $w\in M_2(T)$; (3) $e_1,e_2\in E(T)-\cup_{w\in M_2(T)}E(T_w)$. In case~(1), there exist distinct terminal vertices, say $\ell$ and $\ell'$, of $w$ such that both $e_1$ and $e_2$ lie on the $\ell-\ell'$ path in $T$; thus, $R_e\{e_1, e_2\} \supseteq \{\ell, \ell'\}$ and $h(R_e\{e_1, e_2\}) \ge h(\ell)+h(\ell')=1$. In case~(2), by Lemma~\ref{lemma_tree2}, either $R_e\{e_1,e_2\} \supseteq V(T_w)$ or $R_e\{e_1,e_2\} \supseteq V(T_{w'})$ for some $w'\in M_2(T)-\{w\}$; thus, $h(R_e\{e_1, e_2\}) \ge \min\{h(V(T_w)), h(V(T_{w'}))\}\ge1$ since $w,w'\in M_2(T)$. So, we consider  case~(3). Note that $e_1\in E(T_y)$ for some $y\in M_1(T)$ or $e_1\not\in E(T_z)$ for any $z\in M(T)$; similarly, $e_2\in E(T_{y'})$ for some $y'\in M_1(T)$ or $e_2\not\in E(T_{z'})$ for any $z'\in M(T)$. If $\{e_1, e_2\} \subseteq E(T_v)$ for some $v\in M_1(T)$, then $d(v, e_1)\neq d(v, e_2)$, and there exist distinct $v', v''\in M_2(T)$ such that $v$ lies on the $v'-v''$ path in $T$ and $R_e\{e_1, e_2\} \supseteq V(T_{v'})\cup V(T_{v''})$; thus $h(R_e\{e_1, e_2\})\ge h(V(T_{v'}))+h(V(T_{v''}))\ge 2$. If $\{e_1, e_2\} \not\subseteq E(T_v)$ for any $v\in M(T)$, then there exist distinct $w_1, w_2\in M_2(T)$ such that both $e_1$ (or $y$) and $e_2$ (or $y'$) lie on the $w_1-w_2$ path in $T$; then $d(w_1, e_1)=d(w_1, e_2)$ and $d(w_2, e_1)=d(w_2, e_2)$ imply $e_1=e_2$, contradicting the assumption. So, $R_e\{e_1, e_2\} \supseteq V(T_{w_1})$ or $R_e\{e_1, e_2\} \supseteq V(T_{w_2})$, and thus $h(R_e\{e_1, e_2\}) \ge \min\{h(V(T_{w_1})), h(V(T_{w_2}))\}\ge 1$.~\hfill
\end{proof}

Next, we determine $\edim_f(C_n)$ for $n\ge 3$; we show that $\edim_f(C_n)=\dim_f(C_n)$.

\begin{proposition}
For $n\ge 3$, $\edim_f(C_n)=\left\{
\begin{array}{ll}
\frac{n}{n-1} & \mbox{ if $n$ is odd},\\
\frac{n}{n-2} & \mbox{ if $n$ is even}. 
\end{array}\right.$
\end{proposition}

\begin{proof}
For $n\ge 3$, let $C_n$ be given by $u_0, u_1, \ldots, u_{n-1}, u_0$. Let $g:V(C_n) \rightarrow [0,1]$ be any edge resolving function of $C_n$. For each $i\in\{0, 1, \ldots, n-1\}$, we have the following: (i) if $n$ is odd, then $R_e\{u_iu_{i+1}, u_{i+1}u_{i+2}\}=V(C_n)-\{u_{i+1}\}$ and  $g(V(C_n))-g(u_{i+1}) \ge 1$, where the subscript is taken modulo $n$; (ii) if $n$ is even, then $R_e\{u_iu_{i+1}, u_{i+1}u_{i+2}\}=V(C_n)-\{u_{i+1}, u_{i+1+\frac{n}{2}}\}$ and $g(V(C_n))-g(u_{i+1})-g(u_{i+1+\frac{n}{2}})\ge 1$, where the subscript is taken modulo $n$. In each case, by summing over the $n$ inequalities, we have $(n-1)g(V(C_n))\ge n$ for an odd $n$, and $(n-2)g(V(C_n))\ge n$ for an even $n$. Thus, $\edim_f(C_n) \ge \frac{n}{n-1}$ if $n$ is odd, and $\edim_f(C_n) \ge \frac{n}{n-2}$ if $n$ is even. 

Now, let $h_0$ and $h_1$  be functions defined on $V(C_n)$ as follows: (i) if $n$ is odd, let $h_1(v)=\frac{1}{n-1}$ for each $v\in V(C_n)$; (ii) if $n$ is even, let $h_0(v)=\frac{1}{n-2}$ for each $v\in V(C_n)$. If $n$ is odd, $h_1$ is an edge resolving function of $C_n$ since $|R_e\{e_1, e_2\}| \ge n-1$ for any distinct $e_1, e_2\in E(C_n)$; thus, $\edim_f(C_n) \le h_1(V(C_n))=\frac{n}{n-1}$. If $n$ is even, $h_0$ is an edge resolving function of $C_n$ since $|R_e\{e_1, e_2\}| \ge n-2$ for any distinct $e_1, e_2\in E(C_n)$; thus, $\edim_f(C_n) \le h_0(V(C_n))=\frac{n}{n-2}$.~\hfill
\end{proof}

Next, for the Petersen graph $\mathcal{P}$, we show that $\edim_f(\mathcal{P})=\frac{5}{2}>\dim_f(\mathcal{P})$.

\begin{proposition}\label{edim_frac_petersen}
For the Petersen graph $\mathcal{P}$, $\edim_f(\mathcal{P})=\frac{5}{2}$.
\end{proposition}

\begin{figure}[ht]
\centering
\begin{tikzpicture}[scale=.8, transform shape]

\node [draw, shape=circle, scale=.8] (1) at  (0,2) {};
\node [draw, shape=circle, scale=.8] (2) at  (-1.9, 0.65) {};
\node [draw, shape=circle, scale=.8] (3) at  (-1.36, -1.45) {};
\node [draw, shape=circle, scale=.8] (4) at  (1.36, -1.45) {};
\node [draw, shape=circle, scale=.8] (5) at  (1.9, 0.65) {};
\node [draw, shape=circle, scale=.8] (11) at  (0,0.9) {};
\node [draw, shape=circle, scale=.8] (22) at  (-0.85, 0.3) {};
\node [draw, shape=circle, scale=.8] (33) at  (-0.64, -0.65) {};
\node [draw, shape=circle, scale=.8] (44) at  (0.64, -0.65) {};
\node [draw, shape=circle, scale=.8] (55) at  (0.85, 0.3) {};

\node [scale=1] at (0,2.35) {$u_0$};
\node [scale=1] at (-2.3,0.7) {$u_1$};
\node [scale=1] at (-1.75,-1.5) {$u_2$};
\node [scale=1] at (1.75,-1.5) {$u_3$};
\node [scale=1] at (2.3,0.7) {$u_4$};
\node [scale=1] at (0.38,1) {$w_0$};
\node [scale=1] at (-0.9,0) {$w_1$};
\node [scale=1] at (-0.6,-1) {$w_2$};
\node [scale=1] at (0.6,-1) {$w_3$};
\node [scale=1] at (1,0) {$w_4$};

\draw(1)--(2)--(3)--(4)--(5)--(1);\draw(11)--(33)--(55)--(22)--(44)--(11);\draw(1)--(11);\draw(2)--(22);\draw(3)--(33); \draw(4)--(44);\draw(5)--(55);

\end{tikzpicture}
\caption{Labeling of the Petersen graph.}\label{fig_petersen}
\end{figure}
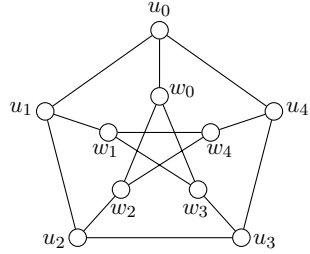

\begin{proof}
Let the vertices of the Petersen graph $\mathcal{P}$ be labeled as in Figure~\ref{fig_petersen}. 

First, we show that $\edim_f(\mathcal{P})\ge \frac{5}{2}$. Let $g:V(\mathcal{P}) \rightarrow [0,1]$ be any edge resolving function of $\mathcal{P}$. For each $i\in\{0,1,2,3,4\}$, we have $R_e\{u_iu_{i+1}, w_{i+2}w_{i+4}\}=\{u_i, u_{i+1}, w_{i+2}, w_{i+4}\}$, and thus $g(u_i)+g(u_{i+1})+g(w_{i+2})+g(w_{i+4}) \ge 1$, where the subscript is taken modulo $5$. By summing over the five inequalities, we have $2g(V(\mathcal{P}))\ge 5$, i.e., $g(V(\mathcal{P}))\ge \frac{5}{2}$. So, $\edim_f(\mathcal{P})\ge \frac{5}{2}$. 

Second, we show that $\edim_f(\mathcal{P})\le \frac{5}{2}$. Let $h:V(G) \rightarrow [0,1]$ be a function defined by $h(v)=\frac{1}{4}$ for each $v\in V(\mathcal{P})$; then $h(V(\mathcal{P}))=\frac{10}{4}=\frac{5}{2}$. It suffices to show that $h$ is an edge resolving function of $\mathcal{P}$. Let $e_1$ and $e_2$ be distinct edges in $\mathcal{P}$. If $e_1$ and $e_2$ are not adjacent in $\mathcal{P}$, then, clearly, $|R_e\{e_1, e_2\}|\ge 4$. So, suppose $e_1$ and $e_2$ are adjacent in $\mathcal{P}$. Since $\mathcal{P}$ is edge-transitive (see~\cite{petersen}), we may assume that $e_1=u_0u_1$; then $e_2\in\{u_1u_2, u_1w_1, u_0u_4, u_0w_0 \}$. If $e_2=u_1u_2$, then $R_e\{e_1,e_2\}=\{u_0, u_2, u_3, u_4, w_0, w_2\}$ with $|R_e\{e_1,e_2\}|=6$; similarly, for each $e_2\in\{u_1w_1, u_0u_4, u_0w_0 \}$, we have $|R_e\{e_1,e_2\}|=6$. So, for any distinct edges $e_1, e_2\in E(\mathcal{P})$, $|R_e\{e_1, e_2\}|\ge 4$ and $h(R_e\{e_1, e_2\}) \ge 4(\frac{1}{4})=1$; thus, $h$ is an edge resoling function of $\mathcal{P}$. So, $\edim_f(\mathcal{P}) \le h(V(\mathcal{P}))=\frac{5}{2}$.~\hfill
\end{proof}

Next, we determine $\edim_f(W_n)$ for the wheel graph $W_n$ of order $n\ge4$; we show that $\edim_f(W_n)>\dim_f(W_n)$ for $n\ge 5$.

\begin{proposition}
Let $W_{n}=K_1+C_{n-1}$ be the wheel graph of order $n \ge 4$. Then 
\begin{equation*}
\edim_f(W_{n})=\left\{
\begin{array}{ll}
\frac{n}{2} & \mbox{ if } n \in\{4,5\},\\
\frac{n-1}{2} & \mbox{ if } n\ge6. 
\end{array}\right.
\end{equation*}
\end{proposition}

\begin{proof}
Let $n\ge 4$. Let $C_{n-1}$ be given by $u_0, u_1, \ldots, u_{n-2}, u_0$, and let $W_n$ be the wheel graph obtained from disjoint union of $C_{n-1}$ and $K_1$ by joining an edge between the vertex $v$ of $K_1$ and each vertex of $C_{n-1}$. Let $g: V(W_n) \rightarrow [0,1]$ be any edge resolving function of $W_n$.

First, suppose $n=4$. Then $\edim_f(W_4)=2$ by Theorem~\ref{dim_frac_graph}(d) and Theorem~\ref{edim_frac_n2}(a). 

Second, suppose $n=5$. For each $i\in\{0,1,2,3\}$, we have the following: (i) $R_e\{vu_i, vu_{i+1}\}=\{u_i, u_{i+1}\}$ and $g(u_{i})+g(u_{i+1}) \ge 1$, where the subscript is taken modulo $4$; (ii) $R_e\{u_iv, u_iu_{i+1}\}=\{v, u_{i+1}\}$ and $g(v)+g(u_{i+1}) \ge 1$, where the subscript is taken modulo $4$. By summing over the four inequalities of (i) and (ii), respectively, we have $2 \sum_{i=0}^{3} g(u_i)\ge 4$ ($\spadesuit$) and $4g(v)+\sum_{i=0}^{3} g(u_i) \ge 4$ ($\maltese$). Now, if we multiply $\frac{3}{2}$ to the inequality ($\spadesuit$) and add it to the inequality ($\maltese$), then we have $4 g(V(W_5)) \ge 10$, i.e., $g(V(W_5))\ge \frac{5}{2}$; thus, $\edim_f(W_5) \ge \frac{5}{2}$. By Proposition~\ref{edim_frac_bounds}, $\edim_f(W_5)=\frac{5}{2}$.  

Next, suppose $n\ge 6$. Note that, for each $i\in\{0,1,\ldots, n-2\}$, $R_e\{vu_i, vu_{i+1}\}=\{u_i, u_{i+1}\}$ and thus $g(u_i)+g(u_{i+1}) \ge 1$, where the subscript is taken modulo $(n-1)$. By summing over the $(n-1)$ inequalities, we have $2g(V(W_n))\ge n-1$; thus, $g(V(W_n)) \ge \frac{n-1}{2}$, which implies $\edim_f(W_n) \ge \frac{n-1}{2}$. Now, let $h:V(G) \rightarrow [0,1]$ be a function defined by $h(v)=0$ and $h(u_i)=\frac{1}{2}$ for each $i \in \{0,1,\ldots, n-2\}$; then $h(V(W_n))=\frac{n-1}{2}$. We show that $h$ is an edge resolving function of $W_n$. Let $e_1, e_2\in E(W_n)$ with $e_1\neq e_2$. If $e_1$ and $e_2$ are not adjacent in $W_n$, then $|R_e\{e_1, e_2\}| \ge 4$. If $e_1$ and $e_2$ are adjacent in $W_n$, then $R_e\{vu_i, vu_j\} \supseteq \{u_i, u_j\}$, $R_e\{u_iu_{i+1}, u_iu_{i-1}\} \supseteq \{u_{i+1}, u_{i-1}\}$, and $R_e\{u_iv, u_iu_{i+1}\} \supseteq \{v, u_{i+1}, u_{i+3}\}$, where the subscript is taken modulo $(n-1)$. In each case, $|R_e\{e_1, e_2\} \cap V(C_{n-1})| \ge 2$, and thus $h(R_e\{e_1, e_2\}) \ge 2(\frac{1}{2})=1$. So, $h$ is an edge resolving function of $W_n$; thus, $\edim_f(W_n) \le h(V(W_n))=\frac{n-1}{2}$. Therefore, $\edim_f(W_n)=\frac{n-1}{2}$ for $n\ge 6$.~\hfill 
\end{proof}

Next, we determine $\edim_f(G)$ when $G=K_{a_1,a_2, \ldots, a_k}$ is a complete $k$-partite graph of order $n=\sum_{i=1}^{k}a_i\ge3$; we show that $\edim_f(G) \ge \dim_f(G)$, where $\edim_f(G) > \dim_f(G)$ if and only if $k\ge 3$ and $a_i=1$ for exactly one $i\in\{1,2,\ldots, k\}$.

\begin{proposition}\label{edim_frac_kpartite}
For $k\ge 2$, let $G$ be a complete $k$-partite graph of order $n\ge 3$. Then 
\begin{equation}
\edim_f(G)=\left\{
\begin{array}{ll}
\frac{n-1}{2} & \mbox{ if $G=K_{1, n-1}$},\\
\frac{n}{2} & \mbox{ otherwise}. 
\end{array}\right.
\end{equation}
\end{proposition}

\begin{proof}
For $k \ge 2$, let $G=K_{a_1, a_2,\ldots, a_{k}}$ be a complete $k$-partite graph of order $n=\sum_{i=1}^{k}a_i\ge 3$, and let $s$ be the number of partite sets of $V(G)$ consisting of exactly one element. Let $V(G)$ be partitioned into $V_1, V_2, \ldots, V_{k}$ with $|V_i|=a_i$ for each $i \in \{1,2,\ldots, k\}$; further, let $a_1 \le a_2 \le \ldots \le a_{k}$. We consider two cases.

\textbf{Case 1: $s \neq 1$.} In this case, $\dim_f(G)=\frac{n}{2}$ by Theorem~\ref{dim_frac_graph}(e). By Theorem~\ref{edim_frac_n2}(a), $\edim_f(G)=\frac{n}{2}$.

\textbf{Case 2: $s=1$.} First, suppose $k=2$; then $G=K_{1,n-1}$. By Proposition~\ref{edim_frac_tree}, $\edim_f(K_{1,n-1})=\frac{n-1}{2}$. 

Second, suppose $k\ge 3$; we show that $\edim_f(G)=\frac{n}{2}$. Let $g:V(G) \rightarrow [0,1]$ be any edge resolving function of $G$. Let $V_1=\{v\}$, $V_2=\{x_1, x_2, \ldots, x_{\alpha}\}$ and $V_3=\{z_1, z_2, \ldots, z_{\beta}\}$, where $\alpha=a_2\ge2$ and $\beta=a_3\ge2$. We note the following: (1) for each $i\in\{1,2,\ldots, \alpha\}$, $R_e\{z_1v, z_1x_i\}=\{v, x_i\}$ and $g(v)+g(x_i)\ge1$; (2) for any distinct $i,j\in\{1,2,\ldots,\alpha\}$, $R_e\{z_1x_i, z_1x_j\}=\{x_i,x_j\}$ and $g(x_i)+g(x_j)\ge 1$. So, we have $\alpha \cdot g(v)+\sum_{i=1}^{\alpha}g(x_i)\ge \alpha$ by summing over the $\alpha$ inequalities from (1), and $(\alpha-1)\sum_{i=1}^{\alpha}g(x_i) \ge {\alpha \choose 2}$ by summing over the ${\alpha \choose 2}$ inequalities from (2); thus, $\alpha(g(v)+\sum_{i=1}^{\alpha}g(x_i)) \ge \alpha+{\alpha \choose 2}$, i.e., $g(V_1)+g(V_2)=g(v)+\sum_{i=1}^{\alpha}g(x_i)\ge \frac{1+\alpha}{2}=\frac{a_1+a_2}{2}$. We also note that, for any distinct $i,j\in \{1,2,\ldots, \beta\}$, $g(z_i)+g(z_j)\ge 1$ by Observation~\ref{obs_twin}(d). By summing over the ${\beta \choose 2}$ inequalities, we have $(\beta-1)\sum_{i=1}^{\beta} g(z_i)\ge {\beta \choose 2}$, i.e., $g(V_3)=\sum_{i=1}^{\beta}g(z_i) \ge \frac{\beta}{2}=\frac{a_3}{2}$. Similarly, if $k\ge4$, then we have $g(V_i)\ge \frac{a_i}{2}$ for each $i\in\{4,\ldots, k\}$. Thus, $g(V(G))=\sum_{i=1}^{k}g(V_i)\ge\sum_{i=1}^{k}\frac{a_i}{2}=\frac{n}{2}$, and hence $\edim_f(G)\ge \frac{n}{2}$. Since $\edim_f(G)\le\frac{n}{2}$ by Proposition~\ref{edim_frac_bounds}, we have $\edim_f(G)=\frac{n}{2}$ for $k\ge 3$.~\hfill
\end{proof}

Next, for the Cartesian product of two paths $P_s \square P_t$, where $s,t\ge2$, we show that $\edim_f(P_s\square P_t)=2=\dim_f(P_s \square P_t)$.

\begin{proposition}\emph{\cite{edim}}\label{edim_grid}
For $s,t \ge 2$, $\edim(P_s \square P_t)=2$.
\end{proposition}

\begin{proposition}
For $s,t \ge 2$, $\edim_f(P_s \square P_t)=2$.
\end{proposition}

\begin{proof}
Let $s,t\ge 2$. Let the vertices of $P_s \square P_t$ be labeled as in Figure~\ref{fig_grid}, and let $W=\{u_{i,j}: i \in \{1,s\} \mbox{ or } j \in \{1,t\}\} \subseteq V(P_s \square P_t)$. We note the following: (i) $R_e\{u_{1,1}u_{2,1}, u_{1,1}u_{1,2}\}=\{u_{i,1}: 2 \le i \le s\} \cup\{u_{1,j}: 2 \le j \le t\}$; (ii) $R_e\{u_{1,t}u_{2,t},u_{1,t}u_{1,t-1}\}=\{u_{i,t}: 2 \le i \le s\} \cup \{u_{1,j}: 1 \le j \le t-1\}$; (iii) $R_e\{u_{s,1}u_{s-1,1}, u_{s,1}u_{s,2}\}=\{u_{i,1}: 1 \le i \le s-1\} \cup \{u_{s,j}: 2 \le j\le t\}$; (iv) $R_e\{u_{s,t}u_{s-1,t}, u_{s,t}u_{s,t-1}\}=\{u_{i,t}: 1 \le i \le s-1\} \cup \{u_{s,j}: 1 \le j \le t-1\}$. If $g: V(P_s \square P_t) \rightarrow [0,1]$ is any egde resolving function of $P_s \square P_t$, then $\sum_{i=2}^{s}g(u_{i,1})+\sum_{j=2}^{t}g(u_{1,j}) \ge 1$, $\sum_{i=2}^{s} g(u_{i,t})+\sum_{j=1}^{t-1}g(u_{1,j}) \ge1$, $\sum_{i=1}^{s-1} g(u_{i,1})+\sum_{j=2}^{t} g(u_{s,j})\ge 1$ and $\sum_{i=1}^{s-1}g(u_{i,t})+\sum_{j=1}^{t-1}g(u_{s,j}) \ge 1$. By summing over the four inequalities, we have $2g(W)\ge 4$. So, $g(V(P_s \square P_t)) \ge g(W)=2$, and thus $\edim_f(P_s \square P_t) \ge 2$. Since $\edim_f(P_s \square P_t) \le 2$ by Observation~\ref{obs_frac}(b) and Proposition~\ref{edim_grid}, we have $\edim_f(P_s  \square P_t)=2$.~\hfill
\end{proof}

\begin{figure}[ht]
\centering
\begin{tikzpicture}[scale=.65, transform shape]

\node [draw, shape=circle] (4) at  (0,4.5) {};
\node [draw, shape=circle] (3) at  (0,3) {};
\node [draw, shape=circle] (2) at  (0,1.5) {};
\node [draw, shape=circle] (1) at  (0,0) {};
\node [draw, shape=circle] (44) at  (1.5,4.5) {};
\node [draw, shape=circle] (33) at  (1.5,3) {};
\node [draw, shape=circle] (22) at  (1.5,1.5) {};
\node [draw, shape=circle] (11) at  (1.5,0) {};
\node [draw, shape=circle] (444) at  (3,4.5) {};
\node [draw, shape=circle] (333) at  (3,3) {};
\node [draw, shape=circle] (222) at  (3,1.5) {};
\node [draw, shape=circle] (111) at  (3,0) {};
\node [draw, shape=circle] (4444) at  (4.5,4.5) {};
\node [draw, shape=circle] (3333) at  (4.5,3) {};
\node [draw, shape=circle] (2222) at  (4.5,1.5) {};
\node [draw, shape=circle] (1111) at  (4.5,0) {};
\node [draw, shape=circle] (44444) at  (6,4.5) {};
\node [draw, shape=circle] (33333) at  (6,3) {};
\node [draw, shape=circle] (22222) at  (6,1.5) {};
\node [draw, shape=circle] (11111) at  (6,0) {};
\node [draw, shape=circle] (444444) at  (7.5,4.5) {};
\node [draw, shape=circle] (333333) at  (7.5,3) {};
\node [draw, shape=circle] (222222) at  (7.5,1.5) {};
\node [draw, shape=circle] (111111) at  (7.5,0) {};

\draw(1)--(2)--(3)--(4);
\draw(11)--(22)--(33)--(44);
\draw(111)--(222)--(333)--(444);
\draw(1111)--(2222)--(3333)--(4444);
\draw(11111)--(22222)--(33333)--(44444);
\draw(111111)--(222222)--(333333)--(444444);
\draw(1)--(11)--(111)--(1111)--(11111)--(111111);
\draw(2)--(22)--(222)--(2222)--(22222)--(222222);
\draw(3)--(33)--(333)--(3333)--(33333)--(333333);
\draw(4)--(44)--(444)--(4444)--(44444)--(444444);

\node [scale=1.3] at (-0.7,0) {$u_{1,1}$};
\node [scale=1.3] at (-0.7,1.5) {$u_{1,2}$};
\node [scale=1.3] at (-0.7,3) {$u_{1,3}$};
\node [scale=1.3] at (-0.7,4.5) {$u_{1,4}$};

\node [scale=1.3] at (1.5,-0.4) {$u_{2,1}$};
\node [scale=1.3] at (2.05,1.75) {$u_{2,2}$};
\node [scale=1.3] at (2.05,3.25) {$u_{2,3}$};
\node [scale=1.3] at (1.5,4.9) {$u_{2,4}$};

\node [scale=1.3] at (3,-0.4) {$u_{3,1}$};
\node [scale=1.3] at (3.55,1.75) {$u_{3,2}$};
\node [scale=1.3] at (3.55,3.25) {$u_{3,3}$};
\node [scale=1.3] at (3,4.9) {$u_{3,4}$};

\node [scale=1.3] at (4.5,-0.4) {$u_{4,1}$};
\node [scale=1.3] at (5.05,1.75) {$u_{4,2}$};
\node [scale=1.3] at (5.05,3.25) {$u_{4,3}$};
\node [scale=1.3] at (4.5,4.9) {$u_{4,4}$};

\node [scale=1.3] at (6,-0.4) {$u_{5,1}$};
\node [scale=1.3] at (6.55,1.75) {$u_{5,2}$};
\node [scale=1.3] at (6.55,3.25) {$u_{5,3}$};
\node [scale=1.3] at (6,4.9) {$u_{5,4}$};

\node [scale=1.3] at (8.2,0) {$u_{6,1}$};
\node [scale=1.3] at (8.2,1.5) {$u_{6,2}$};
\node [scale=1.3] at (8.2,3) {$u_{6,3}$};
\node [scale=1.3] at (8.2,4.5) {$u_{6,4}$};

\end{tikzpicture}
\caption{Labeling of $P_6 \square P_4$.}\label{fig_grid}
\end{figure}
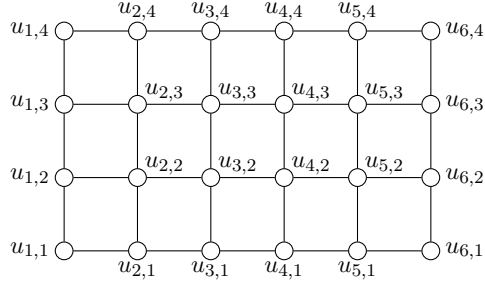

We conclude this section with a few open problems and  a remark showing that $\edim(G)-\edim_f(G)$ can be arbitrarily large.

\begin{question}
Are there graphs $G$ satisfying $\dim(G)=\edim(G)$ and $\dim_f(G)\neq \edim_f(G)$? Are there graphs $H$ satisfying $\dim(H)\neq\edim(H)$ and $\dim_f(H)=\edim_f(H)$?
\end{question}

\begin{question}
Does $\dim(G)>\edim(G)$ imply $\dim_f(G) \ge \edim_f(G)$? Similarly, does $\dim(G)<\edim(G)$ imply $\dim_f(G) \le \edim_f(G)$?
\end{question}

\begin{remark}
There exists a graph $G$ such that $\edim(G)-\edim_f(G)$ can be arbitrarily large. For $n\ge 3$, $\edim(K_n)=n-1$ (see~\cite{edim}) and $\edim_f(K_n)=\frac{n}{2}$ by Proposition~\ref{edim_frac_kpartite}; thus, $\edim(K_n)-\edim_f(K_n)=n-1-\frac{n}{2}=\frac{n-2}{2} \rightarrow \infty$ as $n \rightarrow \infty$. For another example, let $T$ be a tree with exterior major vertices $v_1, v_2, \ldots, v_k$ such that $ter(v_i)=\alpha\ge 3$ for each $i\in\{1,2,\ldots, k\}$, where $k\ge1$. Then $\edim(T)=\dim(T)=(\alpha-1)k$ (see~\cite{edim}) and $\edim_f(T)=\dim_f(T)=(\frac{\alpha}{2})k$ by Proposition~\ref{edim_frac_tree}; thus, $\edim(T)-\edim_f(T)=(\frac{\alpha-2}{2})k \rightarrow \infty$ when $\alpha \rightarrow \infty$ or $k \rightarrow \infty$.
 \end{remark}




\end{document}